\documentclass{article}
\usepackage{amssymb}

\newtheorem{theorem}{Theorem}

\newtheorem{corollary}[theorem]{Corollary}
\newtheorem{criterion}[theorem]{Criterion}
\newtheorem{definition}[theorem]{Definition}
\newtheorem{example}[theorem]{Example}

\newtheorem{lemma}[theorem]{Lemma}

\newtheorem{problem}[theorem]{Problem}
\newtheorem{proposition}[theorem]{Proposition}
\newtheorem{remark}[theorem]{Remark}

\newenvironment{proof}[1][Proof]{\noindent\textbf{#1.} }{\ \rule{0.5em}{0.5em}}
\input{tcilatex}
\begin{document}

\title{Virtual extensions of modules}
\author{by Stephanos Gekas \\
Aristotle University of Thessaloniki, School of Mathematics}
\date{6 March 2017}
\maketitle

\begin{abstract}
In this article we are examining extensions and some basic diagrammatic
properties of modules, in both cases from a new, "virtual" point of view. As
natural background for investigating the kind of problems we are dealing
with, the virtual category of a module $M$ is introduced, having as objects
the submodules of $M$'s subquotients modulo some identifications. In the
case of extensions our approach implies viewing \textquotedblleft
proportionality classes\textquotedblright\ of extensions of (dually, by) a
simple module by (resp. of) another simple as quotients of a certain
quotient (which is in fact a subdirect product) of a projective cover, that
comprises all those classes - or dually as submodules of a comprising
submodule (which is a push-out) of an injective hull. In particular we
become thus able to upgrade the Yoneda correspondence to a bimodule
isomorphism. Basic steps toward the foundation of and investigation into the
theory of Virtual Diagrams are also made here. In particular, the
"virtuality group" $\mathfrak{A}\left( \mathfrak{D}\right) $ of a virtual
diagram $\mathfrak{D}$ of a module $M$ is introduced, generated by the $%
\mathfrak{D}$-visible virtual constituents of $M$, with respect to an
addition that generalizes the one of submodules in a module.
\end{abstract}

\section{\textbf{Introduction} \ }

\ \ \ \textit{\ }

The motivation for this new "virtual", in a sense to be made precise here,
approach to extensions and related "diagrammatic properties" of modules on
the one hand comes from my earlier work \cite{StG1}, whose outset is a
virtual investigation of subdirect group products, on the other hand it has
been necessitated and inspired during my main work toward a theory of
virtual diagrams for modules \cite{StG}. One might therefore consider the
present work as partly inspired by the first and as a transition to the
second of the two mentioned articles.

So we turn now from groups to modules and their extensions from a new, 
\textit{virtual} point of view, which allows us to view and compare them as
a whole families, according to their "embodiment" as well determined
submodules of subquotients (to be called "virtual constituents") in the
structure of a fixed module $M$ - and not just in an abstract way, i.e. up
to isomorphism: The quality of virtuality is crystallized into the notion of
the virtual category $\mathfrak{V}\left( M\right) $ of a module \textit{(or
even of a block, although we do not directly look at it here)}, whose
objects shall be called the \textit{virtual constituents} of $M$. Then we
may embed any indecomposable virtual constituent of Loewy length 2 in a
"partly virtualized" version of an injective hull (of its socle) over an
indecomposable quotient of $M$ or get it as a quotient of some "partly
virtualized" version of a projective cover of its head over some
indecomposable submodule of $M$ (see the generalizing corollary \ref{extr}
of proposition \ref{ipE}), an idea leading to the very important notions of 
\textit{projective and injective extracts} of such 2-fold virtual
constituents of $M$. In section 3 we delve much deeper into the nature of
such non-split constituents, by studying all of their "proportionality"
extension classes, again viewed in two dual ways: either\ comprised
(altogether) in a quotient of a projective cover, which is in fact a
subdirect product, or dually as a submodule of an injective hull, which is
indeed a push-out. 

That analysis allows us also to upgrade the Yoneda correspondence to a
bimodule isomorphism (see Th. \ref{YonHom} \& its corollary).

In that way we win a new insight into the structure of modules, which makes
up the background and puts us well on the way for obtaining the new kind of
diagrams, i.e. ones defined "in virtual terms" (Virtual Diagram), that we
already are introducing here - to further enhance their study (including
proof of existence) in \cite{StG}. The virtual diagrams may somehow, as
suggested in the beginning, be viewed as a parallel extension of some of the
results in \cite{StG1} and their dualization in a category of abelian groups
or of (suitable) modules - a parallelism such as, for example, can be seen
in proposition \ref{ReD} and its corollaries. Based on such virtual
qualities and their natural generalization, we define or rather describe
here what a "virtual diagram" should be, see our Criterion \ref{imp}, so
that we may then begin to study its properties, even before proving its
existence for arbitrary modules in suitable categories. An interesting fruit
of this approach is the "virtuality group" $\mathfrak{A}\left( \mathfrak{D}%
\right) $ of a virtual diagram $\mathfrak{D}$ of a module $M$, generated by
the $\mathfrak{D}$-visible virtual constituents of $M$, with respect to an
addition that generalizes the one of submodules in a module. A "common
enclosure" $\wedge $ on the family of the $\mathfrak{D}$-visible virtual
constituents of $M$, generalizing the votion of intersection of submodules
in a module, is likewise introduced in that set of generators. 

\bigskip \textit{Key words: Virtual diagrams, virtual category, submodule
lattice, (upper/lower) proportional extensions, projective/injective
extracts of virtual constituents,virtuality group }$A\left( \mathfrak{D}%
\right) $\textit{\ of a virtual diagram }$D\left( M\right) $\textit{, Yoneda
correspondence, virtual extension sum.}

\bigskip 

\section{From pull-backs \& push-outs to diagrams}

\bigskip

As we also have thoroughly seen in \cite{StG1}, subgroups of direct products
may be viewed as pull-backs, i.e., fiber products, whose structure
"naturally" conveys diagrammatic depictions of the form $%
\begin{array}{c}
\diagup |\diagdown%
\end{array}%
$ - with $m$ edges in the general case of theorem 34 in \cite{StG1}.
Although we have not given a proper general definition for diagrams, its
suggested use in this case just corresponds to that structure theorem, and
has certainly the special restriction that it refers to a particular
representation of a group $U$ as a subdirect product. It bears,
nevertheless, the basic characteristic that we might expect of any diagram:
Consisting of just two layers (levels), the vertices of the lower one
correspond to (a direct product of) subobjects (subgroups), the vertex on
top is a factor group, namely one that is a factor group in many different
ways according to our general theorem 34. Let us, for our ease, allow
ourselves to call the vertex on top \textit{the head of our depiction of }$U$%
, the direct product of the lower level its \textit{socle}; notice that this
refers only to the particular embedding of $U$\ as a subdirect product.

As expected, in what follows we shall take it for given that any diagram
edge has to represent non-split extensions. This representation will turn
out to be "virtual" for \textit{our} diagrams, in a sense to be specified in
sections 3 and 4 - and then generalized through the Virtual Category for all
diagram edges, as we are going to show in \cite{StG}.

If we, conversely, use the investigated structure of subgroups of direct
products as an inspiration in order to deduce expected properties for such a
basic diagram\textbf{,} then our general theorem 34, our analysis of the
subdirect product structure and, in particular, lemmata 2, 9, 28 and\ 35 in 
\cite{StG1}\ make it clear that:

\begin{lemma}
\label{pbf}\textbf{a.} Any subdiagram of the suggested subdirect product
representation diagram of such a subgroup $U$ comprising just a single edge
(or any proper subset of the set of edges) corresponds to a certain factor
group - but never to a subgroup of $U$. Consequently, there is no proper
subgroup of $U$\ that corresponds to any proper subdiagram containing the
top vertex.

\textbf{b.} The diagrammatic properties of any subdiagram as in (a),
comprising any number of edges, corresponding to a factor group of $U$, are
the same as of the whole diagram of $U$ - i.e., property \textbf{(a)}
"repeats itself".
\end{lemma}

\bigskip Notice that whenever we speak of subdiagrams here, we shall mean
that they are connected (unless otherwise stated) and that they include any
edge of the given one if and only if they also include both its ends.

But there is another major feature to justify taking this kind of simple
diagrams as a major cornerstone for a diagrammatic theory:\ \ Namely, what
makes such a diagrammatic depiction especially interesting and worth
studying for us is its \textbf{virtuality}, in the sense that the multiple
direct factors of the "socle" \textit{also are determined set-theoretically}
(as specific subsets of well defined subsections of $U$, by extending
set-theoretical to a module-theoretical sense) and not just up to
isomorphism - meaning that: \textbf{Their vertices correspond to
well-defined subsets of well-defined subsections.} Notice that this
virtuality could not possibly be claimed just by reference to pull-backs, as
these are only defined up to isomorphism; this is why our first approach has
been through "subgroups of direct products".

\textit{Our main focus with diagrammatic methods shall from now on however
shift from groups to modules and representation theory. We want to move
toward a new kind of diagrams there, one having "virtual properties" in a
sense that generalizes the basic "virtuality" described above. The original
motivation toward the main subject of this article has actually been to
begin understanding and substantializing "virtuality" as well as possible:
then I chose to generalize by considering the more difficult category of
groups instead of that of abelian groups or modules, while viewing it as
very interesting also for its own sake. }

The next natural step would now be to consider the dual case, i.e. push-outs.

Things do however get somewhat complicated, if we attempt to dualize these
ideas in the category of groups: In that category coproducts are namely
given by free products - and push-outs by amalgamated \textit{free}
products. That deviates quite from our original motive - and we shall
therefore not look at them in this article. Of particular interest is,
however, the case of extensions of an arbitrary group $G$ by an abelian
group $A$: then the push-out gives the extension that is (functorially)
induced by homomorphisms $A\rightarrow A^{\prime }$\ of $G$-modules, while
the ones induced by homomorphisms $G^{\prime }\rightarrow G$\ are as usually
(see below) obtained as pull-backs. The push-out mentioned here is in this
case a quotient of the semidirect, rather than the direct, product (see \cite%
[IV 3, exercise 1(b); see also ex. 2]{KB}).

On the contrary, things become much more (dually) analogue to what we have
seen about pull-backs also in the case of push-outs, if we confine ourselves
to the category $\mathfrak{Ab}$ of abelian groups: then the push-out of a
family $S\rightarrow A_{i}$, $i\in I$, of morphisms is obtained as a certain
factor group of their direct sum (coproduct); one may compare this with our
example 92 in \cite{StG1}.

So, in $\mathfrak{Ab}$ we may draw diagrammatic fan-like depictions for
push-outs as well, like the above on pull-backs but "dual" to them, i.e. the
common vertex of all edges shall now be at the bottom, in a form like $%
\begin{array}{c}
\diagdown |\diagup%
\end{array}%
$. \textit{Such a diagrammatic depiction should at this phase be viewed as a
tool to summarize, hold together and overview some facts and arguments, such
as the statements of the last lemma, which we may now dualize in the
following one.}

Speaking of pull-backs and push-outs with corresponding fan-like diagrams,
we must point out that such a "fan"-diagram may in both cases even
degenerate to a single edge; we are also going to investigate this case -
but not even in the category $\mathfrak{Ab}$: From now on we shall be
considering \textbf{modules} instead, by letting a ring (/a $\Bbbk $%
-algebra, f.ex. a group ring) act on abelian groups (/on $\Bbbk $-modules) -
to begin with, generally in a category of (left) Artinian and Noetherian $R$%
-modules, $R$ being a ring with 1: For example, in a category of finitely
generated modules over an Artinian ring.

\bigskip We are now dualizing the above lemma:

\begin{lemma}
\label{Pof}a. In any fan-like push-out module diagram, i.e. of type $%
\begin{array}{c}
\diagdown |\diagup%
\end{array}%
$, in which vertices are virtually determined (\& not just up to
isomorphism), any proper subdiagram comprising a single edge (or any proper
subset of the set of edges), i.e. any "subfan", corresponds to a certain
(virtually determined) proper submodule, which cannot be obtained as a
factor module. This property holds also for any subdiagram of the given of
the same type (: any subfan) - i.e., comprising the common vertex at the
bottom end of the edges.

b. The sum of two such submodules realizes the subfan having as vertices the
union of the two corresponding sets of vertices. This is still true in the
case of "generalized diagrams" of this type (see remark \ref{ddep} below).
\end{lemma}

\bigskip

\begin{remark}
\label{ens}\bigskip Assume that such a fan-like diagram consists of $f+1$\
edges.\ The assertion that any non-empty, proper subset of them corresponds
to a quotient (resp., a submodule) but not to a submodule (resp., not a
quotient module) in particular implies that every edge represents a
non-split extension: To see this for an edge just apply that assertion to
the subset of all the others. Then we reasonably extend this demand to an
assumption, that we already have done from the beginning: Any diagram edge,
whether or not involved in a fan-type subdiagram, represents (realizes!)
some non-split virtual extension module (which is of course a subquotient of
the hosting module $M$); we shall delve deeper into this question in section
3.

It is easy to see that "splitness" (and, therefore also, non-splitness) goes
through the identified (submodules of) subquotients of the virtual category
without problem: Hence the notion of splitness is well defined in the
virtual category.
\end{remark}

\bigskip \bigskip

\begin{remark}
\label{ddep}Notice that the diagrammatic depictions in the preceding 2
lemmata, the previous remark the and even the following Criterion \ref{Crit}
are supposed to be virtual - but they need not be diagrams in the sense we
introduced in \cite{StG}: They are meant in a more general sense, inasmuch
as their vertices do \textit{not necessarily} represent \textit{simple}
subquotients. This kind of diagrams is of course less analytic but more
comprehensive, we shall therefore sometimes use them also in order to
characterize some "types" of diagrams, while we shall be referring to them
as "generalized diagrams". Notice that virtual diagrams in the usual sense
are also a kind of "generalized diagrams". In fact, any "generalized
diagram" represents (in the way we explain below) a sublattice of the
(visible part of) the lattice of submodules of the module $M$.

So, in what follows whenever we use the term "generalized" concerning a
virtual (sub)diagram, we shall mean one where also non-simple modules may be
represented by a vertex. This is actually utilized in the proof of the
existence theorem of a "centrally tuned" diagram (see \cite{StG}).
\end{remark}

\bigskip

\begin{definition}
We call the common vertex in a generalized diagram of the type of lemma \ref%
{pbf}, viewed as a subdiagram of a virtual diagram $D_{M}$\ of a module $M$,
its \textbf{(head) node}; we shall further call the outgoing (generalized)
edges \textbf{legs}, \textbf{if there are no paths} (of any length) \textbf{%
that join them further down on the diagram }$D_{M}$\textbf{\ with other
paths outgoing from the node.} Dually, in a generalized diagram of the type
of lemma \ref{Pof} we shall call the common vertex its \textbf{(basis) node}%
, while the outgoing generalized edges shall be called \textbf{arms}, if
there are no paths that join them with other paths outgoing from the basis
node higher up on the diagram $D_{M}$. Let a common designation for both
arms and legs be \textbf{limbs}. Nodes that are of both kinds above are
called \textbf{combined nodes}. If we cut the node (together with its edge)
off a limb, we get the corresponding \textbf{blunted limb}.
\end{definition}

\bigskip In this language we may summarize the two last lemmata in

\begin{lemma}
\label{comprL}A leg (or a number of legs, together making a "subfan" of the
fan) corresponds to a quotient, a blunted leg (or a set of such ones, out of
the same node) to a submodule.

An arm (or more of them, making out a "subfan")\ corresponds to a submodule,
a blunted arm (or a set of such ones, out of the same node) to a quotient
module.
\end{lemma}

The condition we put on limbs, that they do not touch "the stem" again, is
strictly and absolutely necessary for the following expansion of lemma \ref%
{comprL}, as we are going to demonstrate in proposition \ref{ReD}.

\begin{definition}
We shall call such a (generalized) "fan-like" subdiagram \textbf{a
(generalized) fan}, downward or upward, accordingly.
\end{definition}

So lemma \ref{comprL} is amended to the following:

\begin{lemma}
\label{CL}In a generalized virtual module diagram, going down (by choosing
one or more legs at each head node) means taking quotients, going up (by
choosing one or more arms at each basis node) means taking submodules. That
means that in going down (: factoring out) we cannot "cleave" any upward
fan, in going up (: taking submodules) we may not cleave any downward fan -
where by "cleaving" it must be understood choosing a proper subfan (i.e. a
proper subdiagram including the node, so that it still be a fan).
\end{lemma}

\bigskip\ \bigskip

The natural frame for our work on a given module $K$\ from our point of view
shall be its "virtual category"; we copy here our definition from \cite{StG}:

\begin{definition}
\label{VCat}Given a module $K$, we define its "virtual category" \bigskip $%
\mathfrak{V}\left( K\right) $, with objects all of $K$' s (submodules of)
subsections, identified as such (and NOT up to isomorphism!), on the class
of which we make identifications (also elementwise meant, through fixed
natural isomorphisms) of all pairs of the following types: \textbf{(i)} "of
type f2" $\left\{ \left( A+B\right) /B\text{,}\ A\diagup \left( A\cap
B\right) \right\} $, \textbf{(ii)} "of type f3" $\left\{ A/B\text{,}\
A/C\diagup B/C\right\} $,\ \textbf{(iii)} "of type p", meaning that, given a
canonical epimorphism $p:V\rightarrow V/W$, $V_{0}$, $W$\ submodules of V,
such that $res_{V_{0}}p$ is injective, identify\ $\left\{ V_{0}\text{,}\
p\left( V_{0}\right) \right\} $,\ and \textbf{(iv)} naturally isomorphic\
pairs "of type $s$" $\left\{ \left( \tbigoplus\limits_{i=1}^{n}M_{i}\right)
\diagup \left( \tbigoplus\limits_{i=1}^{n}S_{i}\right) \text{,}\
\tbigoplus\limits_{i=1}^{n}M_{i}/S_{i}\right\} $ of constituents of $K$, as
well as pairs "of type $S$", meaning the following: given subsections $A$, $B
$, $C$ of $K$, with $C\subset B\subset A$, $C$\ viewed as\ a submodule of $B$
is identified with $C$\ viewed as\ a submodule of $A$.\ As morphisms we
accept in this category the maps induced by module homomorphisms between the
virtual constituents.
\end{definition}

\bigskip

As especially interesting special cases of virtually identifiable pairs "of
type $S$" we first mention the case when $B$ is a direct summand of $C$ and,
as a special subcase of this, pairs of the form $\left\{ N\text{, }\varpi
\left( N\right) \right\} $, described by introducing the notion of a
"confinement" $\varpi \left( N\right) $ of a (suitable) subsection $N$, i.e.
some "minimally framed" isomorphic submodule-preimage through canonical
epimorphisms from direct sums, also to be called "confined preimages".

Now we will describe the notion of "confinement":

Let $M=\tbigoplus\limits_{i=1}^{n}M_{i}$ be a certain decomposition of the
section $M$ of a given module $K$\ as a direct sum of indecomposables with
each $M_{i}$ posessing a certain submodule $S_{i}$; consider the natural
isomorphism $\sigma :\left( \tbigoplus\limits_{i=1}^{n}M_{i}\right) \diagup
\left( \tbigoplus\limits_{i=1}^{n}S_{i}\right) \longrightarrow \
\tbigoplus\limits_{i=1}^{n}M_{i}/S_{i}$,\ and the canonical epimorphisms $%
\pi _{i}:M_{i}\longrightarrow M_{i}/S_{i}$. Further, for any subset $J$=$%
\left\{ j_{1},...,j_{s}\right\} $ of $\left\{ 1,...,n\right\} $ define as
well the canonical isomorphism $\sigma _{J}:\left( \tbigoplus\limits_{j\in
J}M_{j}\right) \diagup \left( \tbigoplus\limits_{j\in J}S_{j}\right)
\longrightarrow \tbigoplus\limits_{j\in J}M_{j}\diagup S_{j}$, the direct
sum of canonical epimorphisms$\ \pi _{J}:=\left( \pi _{j_{1}},...,\pi
_{j_{s}}\right) :\tbigoplus\limits_{j\in J}M_{j}\longrightarrow
\tbigoplus\limits_{j\in J}M_{j}/S_{j}$ and, finally, let $p_{J}$\ be the
usual direct sum projection corresponding to the index subset $J$, i.e. $%
p_{J}:\tbigoplus\limits_{i=1}^{n}M_{i}/S_{i}\twoheadrightarrow
\tbigoplus\limits_{j\in J}M_{j}/S_{j}$.

Assume, now, $N$ to be any submodule of $M\diagup \left(
\tbigoplus\limits_{i=1}^{n}S_{i}\right) $. Notice that, if we are \textit{%
"inside of a module }$K$\textit{"}, i.e. if $M$ is a section of a module $K$%
, then $M\diagup \left( \tbigoplus\limits_{i=1}^{n}S_{i}\right) $ and $%
\tbigoplus\limits_{i=1}^{n}M_{i}/S_{i}$ are being "virtually" (i.e., in the
virtual category $\mathfrak{V}\left( K\right) $ of $K$) identified (as pairs
"of type $s$"). $\ \ \ \ \ \ \ \ \ \ \ \ \ \ \ \ \ \ \ \ $

\begin{definition}
\bigskip Let $J$\ be the subset of $\left\{ 1,...,n\right\} $, minimal with
the property that $\sigma \left( N\right) $ is contained in $%
\tbigoplus\limits_{j\in J}M_{j}/S_{j}$ (equivalently, that $res_{N}\left(
p_{J}\circ \sigma \right) $\ be injective). We define "\textbf{the
confinement }of $N$" as, $\varpi \left( N\right) :=\left( \sigma
_{J}^{-1}\circ p_{J}\circ \sigma \right) \left( N\right) $\ \ \ \ (a
submodule of $\left( \tbigoplus\limits_{j\in J}M_{j}\right) \diagup \left(
\tbigoplus\limits_{j\in J}S_{j}\right) $).
\end{definition}

\bigskip In the following we shall just be using the notation $\varpi \left(
N\right) $\ for it, without further notification whenever the context is
clear. Take care of the important fact, that this preimage depends on the
choice of a decomposition $M=\tbigoplus\limits_{i=1}^{n}M_{i}$ into
indecomposables, or at least of the direct summand $\tbigoplus\limits_{j\in
J}M_{j}$, unless all the indecomposables $M_{i}$, $i=1,...,n$, are
non-isomorphic, in which case the definition still remains unambiguous even
without any further specification.\ Anyway, even to the extent that the
confinement of $N$ as a submodule of $M\diagup \left(
\tbigoplus\limits_{i=1}^{n}S_{i}\right) $ may be dependent on the particular
decompositions $\tbigoplus\limits_{i=1}^{n}M_{i}$\ and\ $\tbigoplus%
\limits_{i=1}^{n}S_{i}$, we are always going to use it in the context of
some given (/chosen) such decompositions. This, however, becomes\ really
relevant in the continuation of this present work into \cite{StG}.

\bigskip Notice that the direct sums that are relevant in our context here
are just ones that appear as sections of the module: general direct sums are
not definable in the Virtual Category of a module.

\bigskip

In the frame of the virtual category $\mathfrak{V}\left( M\right) $ of a
module $M$\ (see also \cite{StG}) we need a special ordering "$\lessdot $",
defined as follows:

\begin{definition}
\label{VO}In the frame of the virtual category $\mathfrak{V}\left( M\right) $
of $M$ define on the family of submodules of subquotients of $M$\ the
ordering "$\lessdot $", generated by the stipulation that "less than" mean
to be a submodule of a (sub)quotient of, in a non-split way (i.e. not as a
direct summand); we shall accordingly speak of "slimmer" or, on the
contrary, of "thicker" subquotients, especially referring to indecomposable
ones.
\end{definition}

\bigskip We shall occasionally also use the term \textit{"to enclose"} (or 
\textit{"be enclosed by"}) w.r.t. this ordering - and the symbol "$\widehat{=%
}$" to denote equality in this category - i.e. \textit{"in virtual terms"}.
A general term to descrive the objects of the virtual category $\mathfrak{V}%
\left( M\right) $ of $M$ shall be "\textbf{virtual constituents} of $M$".

\begin{definition}
For $0\leq i<j\,\,$, any indecomposable summand of the "radical section" $%
rad^{i}M/rad^{j}M$ shall be called "$\left\{ i,j\right\} $\textbf{-pillars}%
", of \textbf{height }(: Loewy length) at most $j-i$; the least possible of
such heights for the given pillar is its actual height.

A pillar with no other isomorphic pillars in the same radical section shall
be called a \textbf{single pillar.} An $\left\{ i,j\right\} $\textbf{%
-colonnade of rank r, r\TEXTsymbol{>}1,} is a maximal direct sum of \ $r$
isomorphic $\left\{ i,j\right\} $-pillars, the maximality referring to $r$,
in the sense that the sum is not properly contained in a direct sum of more
than $r$\ $\left\{ i,j\right\} $-pillars. A \textbf{specification} \textbf{%
of a colonnade} is any choice of specific pillars, a specification that is $%
\mathfrak{D}$-visible ($\mathfrak{D}$-manifest) is necessarily unique for $%
\mathfrak{D}$.

A pillar $B$ is "dominated" $\left( \leq \right) $ by another one $A$ (or we
may alternatively say that $A$ \textbf{overcoats} $B$),\ if there exists a
specification of each of them such that any pillar of the specification of
the first ($B$) is enclosed (that is, in the sense of "$\lessdot $") by a
pillar of the specification of the latter ($A$). In this ordering we shall
be epsecially interested in "maximally/minimally dominating pillars".
\end{definition}

\bigskip

\textbf{We shall say that an oriented graph }$D$\textbf{\ (without loops or
multiple edges) gives a virtual diagram for a module }$M$\textbf{\ if the
following conditions are satisfied:}

\begin{criterion}
\label{Crit}We demand the following of a virtual diagram $\mathfrak{D}$ for
a module $M$:

(i) Virtual Correspondence: All vertices are virtually determined - i.e., as
well-defined (simple) constituents in the virtual category $\mathfrak{V}%
\left( M\right) $\ of $M$ (i.e. specifically placed "inside $M$" \& not just
up to isomorphism).

(ii) Lemma \ref{CL} and remark \ref{ens} (in particular, vertices represent
non-split extensions) are satisfied \textbf{(hence also their consequences
that are demonstrated in the rest of this subsection)}.

(iii) \textit{By a subdiagram of }$D$\textit{\ we shall mean a subgraph of
it, with the same virtual correspondence of the remaining vertices. SATIETY:
There is no diagram satisfying (i) and (ii) of which }$\mathfrak{D}$\textit{%
\ is a proper subdiagram.} \ 

(iv) A "centrally tuned" diagram must be optimal, in the sense that it
realizes some specification of any maximal colonnade (see \cite{StG}).
\end{criterion}

\begin{definition}
"Realizability" of a connected subdiagram of a module diagram $\mathfrak{D}%
\left( M\right) $\ amounts to the existence of a submodule of a virtual
subquotient (i.e., existence of a virtual constituent) $Q$ of the module,
which corresponds to that subdiagram (as a whole) in the given virtual
diagrammatic depiction\ $\mathfrak{D}$. Conversely, $Q$\ shall then be
called a $\mathfrak{D}$-visible (or $\mathfrak{D}$-manifest) virtual
constituent.
\end{definition}

From now on we shall mainly be interested in and referring to centrally
tuned diagrams.

\begin{lemma}
\bigskip \bigskip\ If $C\subset B\subset A$ are virtual subfactors, then $%
A\diagup B\lessdot A\diagup C$.
\end{lemma}

\begin{proof}
Because in the virtual category $A\diagup B=\left( A\diagup C\right) \diagup
\left( B\diagup C\right) $, i.e. $A\diagup B$ is a quotient of $A\diagup C$.
\end{proof}

\begin{proposition}
\label{StP}Given a module $M$\ and a virtual diagram $D_{M}$ of it, we may
get quotients by going down as in lemma \ref{CL}, thus truncating legs, and
then get to realizable submodule of that quotient by\ continuing reversely,
by going up, thus truncating arms; we would get to the same object of the
virtual category $\mathfrak{V}\left( M\right) $ by going first up, then down
- or even by shuffling steps up with steps down in any possible manner (but
toward the same final subdiagram, of course), as long as we end up in the
same subdiagram of $D_{M}$.
\end{proposition}

\begin{proof}
\bigskip It is quite clear from their definition that generalized legs and
arms are realizable subdiagrams.

As for blunted legs, they are "hanging free" and they represent submodules
of $M$.

Now the identifications that lay the fundament for the definition of the
virtual category (see \ref{VCat}) imply that \ in going down, when we are in
several steps actually factoring blunted legs out, that in a sense are
"hanging freely" from their head nodes, in the virtual category this amounts
to just factoring out their direct sum, just in one step!
\end{proof}

\bigskip Given a module $M$\ with a virtual diagram $D_{M}$,\ it is
reasonable to ask, also in view of last proposition, when a subdiagram is 
\textit{realizable}, meaning that it corresponds to an object of the virtual
category $\mathfrak{V}\left( M\right) $.

Now we are ready to prove the following:

\begin{proposition}
\label{ReD}Given a module $M$\ and a virtual diagram $\mathfrak{D}_{M}$ of
it, a connected subdiagram $D$\ of that is realizable if and only if there
exists no generalized (in the sense or remark \ref{ddep}) vertical (i.e.
monotone) path connecting two of its vertices at different radical layers
and containing at least one vertex outside $D$.
\end{proposition}

\begin{proof}
\textbf{"}$\Longrightarrow $\textbf{":} \bigskip Let $l$\ denote the
"vertical length" of the subdiagram $D$; then we may actually restrict our
attention to the corresponding radical section $rad^{s}M\diagup rad^{s+l}M$\
of $M$.\ We may further make the following two assumptions:

(i) The only vertices in such a path $T$, of length $l\left( T\right) $,
that also belong to $D$, are its endpoints, and (ii) There exists no other
path of length less than $l\left( T\right) $, with the same property and
including at least one of $T$'s intermediate vertices: Because we can then
substitute $T$\ with another one, may be in several steps (for (ii)), having
these two properties.

Call $A$, resp. $B$, the two virtual irreducible subquotients of $M$, that
correspond to the endpoints of the path $T$, respectively lying, say, on the 
$\kappa $-th and $\left( \kappa +\lambda \right) $-th radical\ layer of $M$,
probably with $s+1\leq \kappa $, $\kappa +\lambda \leq s+l$, $\lambda \geq 2$%
\ .

We define as $S=S_{M}\left( T\right) $\ the indecomposable direct summand of
the appropriate (least) radical section $rad^{\kappa -1}M\diagup rad^{\kappa
+\lambda }M$\ of $M$, containing all the virtual simple subquotients
corresponding to the vertices of the path $T$\ from $A$ to $B$; notify the
fact that these indeed lie in the same such indecomposable summand (just
denoted $S$), or we would get a contradiction to the existence of the path $%
T $ and the non-splitness of the extensions represented by its edges. This $%
S $\ has on the one hand to be visible in $D_{M}$, as the \textit{uniquely
determined} such indecomposable summand "containing" a $D_{M}$-path, on the
other the subgraph $D_{S}$ of $D_{M}$,\ that represents it, must be a
connected component of the section of $D_{M}$\ corresponding to $rad^{\kappa
-1}M\diagup rad^{\kappa +\lambda }M$; but then assumptions (i) and (ii)
above imply that $D_{S}$ must have a generalized diagrammatic depiction of
the form: $S=%
\begin{array}{c}
A \\ 
| \\ 
K \\ 
| \\ 
B%
\end{array}%
$\ for some (non-simple, in general) subquotient $K$. \ \ 

Assume first that $D$\ is realizable by a subquotient $N$ and denote by $%
Q:=S_{N}\left( A,B\right) $\ the "slimmest" (i.e. the minimal according to "$%
\lessdot $") among those indecomposable subquotients of $N$, that are
thicker than (: "contain") both $A$\ and $B$; it is not difficult to see
that the family of the indecomposable subquotients of $N$, that are thicker
than both $A$\ and $B$, is inductively ordered with respect to the opposite
of "$\lessdot $": Because any totally ordered subset has to be finite, since
at any step of a (descending) chain at least one of its virtual composition
factors is taken away; but then we reach at a lower bound of the chain with
the last step. Hence there exists a minimal such indecomposable $Q$,
according to Zorn's Lemma, which then has to be unique, because it is
minimal (slimmest) as an indecomposable that is thicker than the virtuals $A$%
\ and $B$ (which in this case is easily seen, when combined with its
minimality, to simply mean that $A$\ must be a quotient and $B$ a submodule
thereof). \textit{Its uniqueness now guarantees its }$D_{M}$\textit{%
-visibility}. Then $Q=%
\begin{array}{c}
A^{\prime } \\ 
| \\ 
L \\ 
| \\ 
B^{\prime }%
\end{array}%
$\ \textit{(a generalized diagram)} for some subquotient $L$, where $A$\ is
contained in $Soc\left( A^{\prime }\right) $, $B$\ is contained in $Hd\left(
B^{\prime }\right) $. Assume that the radical span of $B^{\prime }$\ is from
the $\left( \kappa +\lambda \right) $-th to the $\left( \kappa +\lambda +\mu
\right) $-th radical\ layer of $M$. Take the preimage $S^{\prime }$ of $S$\
under the canonical epimorphism $rad^{\kappa -1}M\diagup rad^{\kappa
+\lambda +\mu }M\twoheadrightarrow rad^{\kappa -1}M\diagup rad^{\kappa
+\lambda }M$; define then the module $\widetilde{N}=Q+S^{\prime }$ (a sum of
submodules of $rad^{\kappa -1}M\diagup rad^{\kappa +\lambda +\mu }M$),\
which must then be indecomposable and have a generalized diagrammatic
depiction of the form $%
\begin{array}{ccc}
& A^{\prime } &  \\ 
\diagup &  & \diagdown \\ 
L_{{}}\frame{} &  & \frame{}_{{}}K \\ 
\diagdown &  & \diagup \\ 
& B^{\prime } & 
\end{array}%
$ in the sense of the two previous lemmata.

But then by successive application of the two lemmata \ref{pbf} \& \ref{Pof}
we get for the left "column" $%
\begin{array}{c}
A^{\prime } \\ 
| \\ 
L \\ 
| \\ 
B^{\prime }%
\end{array}%
$, corresponding to (: realized by) the virtual subquotient $Q$, that it
must be both, a quotient and a submodule of $\widetilde{N}$: So, $Q$ has to
be a direct summand of $\widetilde{N}$, a contradiction to the fact that
both $%
\begin{array}{l}
A^{\prime } \\ 
| \\ 
K%
\end{array}%
$\ and $%
\begin{array}{l}
K \\ 
| \\ 
B^{\prime }%
\end{array}%
$\ in the diagram are non-split. Therefore can $D$\ not be realizable.

\textbf{"}$\Longleftarrow $\textbf{":} \textit{Assume, conversely, that
there exists no such path }$T$\textit{.}

That implies that any path outgoing from a vertex of $D$\ does not return to 
$D$. Truncate the parts of these paths, that do not lie inside the radical
section of $M$ containing $D$. All of them together with $D$ generate a
diagram, which is clearly a connected component of the appropriate radical
section of $M$, hence it is realizable as a direct summand $X$ of that
section. But then we can cluster these downward, resp. upward directed paths
up into disconnected (outside $D$) thicker branches of two kinds: down-going
(generalized legs) and up-going (generalized arms). We then use proposition %
\ref{StP} to get to $D$\ as a virtual subquotient of $X$.
\end{proof}

\bigskip The next corollary refers to some definitions given in \cite{StG},
which we do not repeat here, since this corollary is not used anywhere else
in this article.

\begin{corollary}
\bigskip (i) Let $N$ be a module of simple head $A$\ and simple socle $B$,
with a virtual diagram that is a (possibly) composite $\left( m,l\right) $%
-shell ($m>1$, $l>1$) - or even a thick hybrid obtained via such a diagram
(see definitions in \cite{StG}).\ There can be no realizable proper
subdiagrams that contain both $A$\ and $B$.

(ii) If we have such a (hybrid, composite) shell subdiagram as in (i) in the
virtual diagram of a module $M$, so that its only vertices that also are
ends of edges outside the shell are $A$\ and $B$, also then there can be no
realizable proper subdiagrams that contain both $A$\ and $B$.
\end{corollary}

\begin{proof}
\bigskip (i) is a direct consequence of the proposition. As for (ii), it is
clear that such a subdiagram corresponds to a direct summand $V$ of the
appropriate radical section of\ $M$, as it is a connected component of the
subdiagram corresponding to that section - i.e. it realizable; then use (i).
\end{proof}

\begin{corollary}
\bigskip There can be no subdiagram of the type $%
\begin{array}{cc}
A & \searrow _{{}} \\ 
\downarrow & _{{}}L \\ 
B & \swarrow ^{{}}%
\end{array}%
$\ contained in a virtual module diagram. The same remains true if we get to
a generalized subdiagram, by allowing $L$\ be even non-simple.
\end{corollary}

\begin{proof}
\bigskip By allowing $L$\ to be non-simple, if such a subdiagram existed, it
would have been realizable, as a connected component of a radical (or
socle!) series section - call that virtual subquotient $N$.\ The subdiagram $%
\begin{array}{c}
A \\ 
\downarrow \\ 
B%
\end{array}%
$\ has to be realizable, say as a virtual $C$, according to Criterion \ref%
{Crit}, (i). But then $C$\ must be both a quotient and a submodule of $N$,
implying that it is a direct summand of $N$, hence also that $L$\ is such
one, contrary to the meaning of the two edges attached to it.
\end{proof}

\bigskip Since there can be no subdiagrams of the type of the last
corollary, the proposition yields:

\begin{corollary}
A connected subdiagram of radical length 2 is always realizable.
\end{corollary}

Given two subdiagrams $\mathfrak{D}_{A}$, $\mathfrak{D}_{B}$ of $\mathfrak{D}%
_{M}$, we shall understand as their "union" the subdiagram of $\mathfrak{D}%
_{M}$, that consists of the two but also \textit{including any single edges
joining vertices of the one to vertices of the other.} 

\begin{corollary}
\label{vsce}\bigskip In the notation of the proposition, let there be given
two realizable subdiagrams $\mathfrak{D}_{A}$, $\mathfrak{D}_{B}$ of $%
\mathfrak{D}_{M}$ (with the two former realizing the virtual contituents $A$
and $B$), whose overlap is non-empty and which are realized by the virtual
constituents $A$ and $B$\ of $M$. Call $\mathfrak{D}_{K}$  the intersection
and $\mathfrak{D}_{C}$ the union of the two subdiagrams $\mathfrak{D}_{A}$, $%
\mathfrak{D}_{B}$. Then $\mathfrak{D}_{K}$ is realizable - say, by $K$. If
we further assume for the two subdiagrams, that there exist no generalized
(in the sense of remark \ref{ddep}) monotone paths outside them (i.e., going
through some simple virtuals, that do not lie in their union) joining $%
\mathfrak{D}_{A}$ and $\mathfrak{D}_{B}$, then $\mathfrak{D}_{C}$ is also
realizable by some virtual constituent $C$ of $M$, on the condition that one
of the following holds: (i) The overlap $\mathfrak{D}_{K}$ of $\mathfrak{D}%
_{A}$ and $\mathfrak{D}_{B}$ is non-empty, (ii) the union $\mathfrak{D}_{C}$
contains at least one new edge (i.e., not contained in either $\mathfrak{D}%
_{A}$ or $\mathfrak{D}_{B}$) or, finally, (iii) $A$ and $B$ may be viewed as
direct summands of some virtual constituent of $M$.

It is then clear that $K$ is the common enclosure of $A$ and $B$, for which
we shall be writing $A\wedge B$.\ We shall further call the virtual object $C
$ just defined the \textbf{virtual sum} of $A$ and $B$ - and write $A%
\widehat{+}B$ or $A\vee B$; i.e., $A\widehat{+}B=C$.

On the other hand, if such a monotone path outside the two realizable
subdiagrams existed, then their union, viewed as a subdiagram of $\mathfrak{D%
}_{M}$, is not realizable. \ \ \ \ \ \ \ 
\end{corollary}

\bigskip We further extend the definition of virtual sum to the case when
there (apparently...) is no overlap of the realizable subdiagrams $\mathfrak{%
D}_{A}$, $\mathfrak{D}_{B}$ of $\mathfrak{D}_{M}$, but their corresponding
virtual constituents $A$ and $B$ may also be considered as submodules of the
same constituent: then we define again the virtual sum $A\widehat{+}B$ as
the virtual object corresponding to the sum of the submodules in the virtual
classes of $A$ and $B$. Remark that there may not exist any monotone paths
outside them joining $\mathfrak{D}_{A}$ and $\mathfrak{D}_{B}$, since that
would easily give a contradiction to lemma 1a for the submodule property of
the virtual to the higher end of such a path. Notice further that, also in
this case, in some sense the intersection of the subdiagrams $\mathfrak{D}%
_{A}$ and $\mathfrak{D}_{B}$ is "trivial" (corresponding to their common,
also virtually, trivial submodule) but non-empty, a point of view that
allows us to subdue this to the general case too. Corresponding to this
point of view, we shall in this case define $A\wedge B$ to be the
appropriate virtual trivial module (i.e., the one that may be considered as
a submodule of $A$ and $B$). \ \ \ 

With this extension of the definition of virtual sum, we may now further
observe that also in the case that the $\mathfrak{D}_{M}$-realizable virtual
constituents $A$ and $B$ may be considered as submodules of the same
constituent with some non-trivial intersection, their usual sum as such
ones, respectively their intersection, coincides with the just given
definition of virtual sum, respectively their common enclosure - so that
both new definitions may be viewd as generalizations. Especially
interesting, as an easy exercise, is to see how these new definitions apply
on subfans of fans (representing $\mathfrak{D}_{M}$-realizable virtual
constituents). Remark that $A\widehat{+}A\widehat{=}A$\ for any $\mathfrak{D}%
_{M}$-realizable virtual constituent $A$.\ \textit{Whenever the context is
clear and there is no risk of misunderstanding, we may also just write }$+$%
\textit{\ and }$=$\textit{, rather than }$\widehat{+}$\textit{\ and }$%
\widehat{=}$\textit{.}

The following should now be quite clear:

\begin{lemma}
If the virtual constituents $A$ and $B$\ of $M$ are $\mathfrak{D}_{M}$%
-visible, say through the subdiagrams $\mathfrak{D}_{A}$ and $\mathfrak{D}%
_{B}$, but they may neither (by virtual-class representatives) be considered
as submodules of the same virtual constituent nor the common enclosure of
their diagrams $\mathfrak{D}_{A}$ and $\mathfrak{D}_{B}$ is non-empty, then
the union of the diagrams $\mathfrak{D}_{A}$ and $\mathfrak{D}_{B}$ is not
realizable.
\end{lemma}

After this preparation we proceed to also define some formal objects $A%
\widehat{+}B$ even in the case of the preceding lemma, or in the last case
of corollary \ref{vsce}, although this object is no virtual constituent of $M
$, i.e. not any object of the category $\mathfrak{V}\left( M\right) $. To
this end, we stipulate that also the formal object $A\widehat{+}B$ shall
"correspond" to the union of the subdiagrams $\mathfrak{D}_{A}$ and $%
\mathfrak{D}_{B}$. Take account of the important remark that, it is very
well possible, that we have some $\mathfrak{D}_{M}$-visible virtual
constituents $A_{i}$, $i=1,2,3$, of $M$, with\ $A_{1}\widehat{+}A_{2}$ not
being a virtual object, all the while$\ A_{1}\widehat{+}A_{2}\widehat{+}%
A_{3} $\ is!\ \ 

Thus are we lead to define an abelian group $\left( \mathfrak{A}\left( 
\mathfrak{D}_{M}\right) \text{, }\widehat{+}\right) $, with respect to this
extended virtual sum, hence to be called the virtuality group of the virtual
diagram $\mathfrak{D}_{M}$ of $M$, whose formal generators are all the $%
\mathfrak{D}_{M}$-visible virtual constituents of $M$, with relations all
possible relations $A\widehat{+}B=C$, where $A$, $B$ and\ $C$ are $\mathfrak{%
D}_{M}$-visible constituents of $M$.

Then any element of $\mathfrak{A}\left( \mathfrak{D}_{M}\right) $ is by
definition a finite "$\widehat{+}$"-sum $\sum_{i}A_{i}$\ of some $\mathfrak{D%
}_{M}$-visible virtual constituents $A_{i}$ of $M$; it is an easy routine
procedure to prove, that this may in a unique way be written as a sum of a
minimal number of virtual constituents (in our case probably meaning the
same as $\mathfrak{D}_{M}$-realizable virtual constituents), which shall
then be called \textbf{the reduced form} of the sum; call that minimal
number \textbf{the reduced length} of the given element of $\mathfrak{A}%
\left( \mathfrak{D}_{M}\right) $.

All this discussion makes it quite clear that the following lemma is true:

\begin{lemma}
An element of $\mathfrak{A}\left( \mathfrak{D}_{M}\right) $\ corresponds to
a virtual (and $\mathfrak{D}_{M}$-visible) constituent of $M$, if and only
if its reduced length is 1. \ 
\end{lemma}

\ 

Now we want to introduce some more terminology on virtual module diagrams,
also inspired by, among others, \cite{BC}.

Our virtual diagrams shall be considered directed - "downwards". By setting $%
X=\left\{ x_{1},...,x_{t}\right\} $\ for the set of vertices, we shall
denote an edge going from $x_{i}$ to $x_{j}$ by $e\left( x_{i},x_{j}\right) $%
; in that case we shall write the relation $x_{j}<x_{i}$.\ We denote by "$%
\leq $" the ordering relation generated on $X$ by all such inequalities.

Next we introduce a topology on the set of subdiagrams, generated by the
following "open" subdiagrams: Namely any subdiagram $X^{\prime }$ having the
properties that, (i) an edge of $X$ is included in it iff both its endpoints
are, and (ii) whenever $x\in X^{\prime }$\ and $y<x$, $y$\ also\ belongs to $%
X^{\prime }$. Clearly the open subdiagrams (and subsequently also the closed
ones, for which we likewise require condition (i)) are completely determined
by their sets of vertices, hence we may also identify them through those:
This makes the thus generated topology on the set of subdiagrams quite
apparent.

But then we can almost immediately see that the comprehensive lemma \ref{CL}%
\ above is equivalent to the following:

\begin{lemma}
\label{ocSD}Open subdiagrams are realizable as submodules, closed ones as
quotient modules of $M$.
\end{lemma}

Given that any object of the virtual category $\mathfrak{V}\left( M\right) $
of $M$ is a submodule of a subquotient, we come readily to the following
conclusion:

\begin{proposition}
A subdiagram of a virtual diagram of the module $M$\ is realizable if we can
get to it by a (finite) succession of steps of, alternately, taking open and
closed subdiagrams (: subsets of vertices), each time meant in the relative
topology of the last step.
\end{proposition}

\bigskip If we combine lemma \ref{Pof}(b) with lemma \ref{CL}, we come
easily to the following generalization of \ref{Pof}(b) - which has already
been mentioned during our above discussion around the introcuction of
sirtual sums:

\begin{lemma}
For any $D_{M}$-visible submodules $V_{1}$, $V_{2}$ of $M$, their sum $%
V_{1}+V_{2}$ is realized by the open subdiagram having as set of vertices
the union of those of their respective open subdiagrams. Furthermore the
intersection $V_{1}\cap V_{2}$ is realized by the open subdiagram
corresponding to the intersection of them.
\end{lemma}

\bigskip We are now ready to turn to the lattice $\left( \mathcal{L}\left(
M\right) ,\vee ,\wedge \right) $ of submodules of $M$, with lattice
operations induced by the submodule-ordering "$\subseteq $", meaning here
that "$\vee $" (join) is "$+$" and "$\wedge $" (meet) is "$\cap $".

This last lemma shows that the $D_{M}$-visible submodules for a virtual
diagram $D_{M}$ form a sublattice $\mathcal{L}_{D}\left( M\right) $ of\ $%
\mathcal{L}\left( M\right) $.\ 

\begin{remark}
It is not so difficult to show that, in case the lattice $\mathcal{L}\left(
M\right) $\ is distributive, in which case the module $M$\ is called \textbf{%
distributive}, then we actually have $\mathcal{L}_{D}\left( M\right) =%
\mathcal{L}\left( M\right) $.\ In connection to that it is crucial to take
an important result of Victor Camillo (\cite[Theorem 1]{ViC}) into account,
stating the following:

\textit{"An }$R$\textit{-module }$M$\textit{\ is distributive if and only if
for every submodule }$N$\textit{, }$M/N$\textit{\ has square-free socle."}\ 
\end{remark}

\ 

Distributivity is, no less in view of the above result, a very crucial
notion up to the achievement of an existence theorem for a virtual diagram
of a given module (see \cite{StG}).

\begin{remark}
Inclusion of submodules being a special case of "$\lessdot $" in $\mathfrak{V%
}\left( M\right) $, by remembering lemma \ref{ocSD} we realize that
determining the lattice $\mathcal{L}_{D}\left( M\right) $\ of $D_{M}$%
-visible submodules\ is equivalent to determining the "$\lessdot $"-induced
lattice of closed subdiagrams of $D_{M}$, or equivalently (by \ref{ocSD}) of
visible quotient modules (notice that a submodule is $D_{M}$-visible iff the
quotient by it is $D_{M}$-visible, their corresponding $D_{M}$-subdiagrams
being complementary). This is immediately seen to be equivalent to the dual
lattice $\mathcal{L}_{D}^{\prime }\left( M\right) $\ of $\mathcal{L}%
_{D}\left( M\right) $; its meet (say of $M/A$ and$\ M/B$) is $M/\left(
A+B\right) $, its join is defined as $M/\left( A\cap B\right) $.
\end{remark}

We could also get analogue results by starting with a virtual diagram of $M$%
, that is fixed through its socle rather than its radical series. Their
interrelation is also studied in \cite{StG}.

We shall need some new notions in what follows, for the preparation of which
the following lemma is intended:

\begin{lemma}
\label{ssQR}\bigskip \-Let $R$ be an artinian ring, $M$ a finitely generated 
$R$-module, $N\subset M$\ a submodule and denote by $J$ the Jacobson radical
of $R$; assuming that we have a virtual socle as well as a virtual radical
series of $M$,\ the following are true:

(i) The socle series of $N$\ is obtained by intersecting $N$\ with the given
virtual socle socle series of $M$.

(ii) For any $R$-epimorphism $\phi :M\twoheadrightarrow M^{\prime }$, $\phi
\left( rad^{\kappa }M\right) =rad^{\kappa }\phi \left( M\right) =rad^{\kappa
}M^{\prime }$. In particular, by taking as $\phi :M\twoheadrightarrow
M\diagup N$, the canonical epimorphism, we get the dual to (i): $rad^{\kappa
}\left( M\diagup N\right) \widehat{=}rad^{\kappa }M\diagup \left( N\cap
rad^{\kappa }M\right) $.

(iii) $rad^{\kappa }\left( M\diagup \left( N\cap rad^{\kappa }M\right)
\right) \widehat{=}rad^{\kappa }\left( M\diagup N\right) $\ \ 
\end{lemma}

\begin{proof}
\bigskip (i) The proof can be found in \cite[Lemma I 8.5(i)]{PL}; however we
have to remark that not just the proof but even the statement makes only
sense in general in the frame of the virtual category $\mathfrak{V}\left(
M\right) $, because otherwise it becomes impossible whenever there are
multiple isomorphic irreducibles on the same layer and $N$\ encloses (in the
sense of "$\lessdot $") some (but not all) of them.

(ii) Recall that $rad^{\kappa }M=J^{\kappa }M$, whence $\phi \left(
rad^{\kappa }M\right) =\phi \left( J^{\kappa }M\right) =J^{\kappa }\phi
\left( M\right) =J^{\kappa }M^{\prime }=rad^{\kappa }M^{\prime }$. Then $%
rad^{\kappa }\left( M\diagup N\right) =rad^{\kappa }\phi \left( M\right)
=\phi \left( rad^{\kappa }M\right) =$

$=\left( N+rad^{\kappa }M\right) \diagup N\widehat{=}rad^{\kappa }M\diagup
\left( N\cap rad^{\kappa }M\right) $, according to identification f2 in the
definition of the virtual category $\mathfrak{V}\left( M\right) $.

(iii) It is quite clear that $M\diagup \left( N\cap radM\right) \widehat{=}%
\left( M\diagup N\right) \widehat{+}V$ \ for some completely reducible
module $V$, sitting on the head of $M$ (hence no part of $V$\ is enclosed by 
$radM$). From that relation it follows that $rad\left( M\diagup \left( N\cap
radM\right) \right) =rad\left( M\diagup N\right) $. Then the inductive
argument proceeds by substituting above $radM$\ for $M$\ \& $N\cap radM$ for 
$N$\ - and so on.
\end{proof}

The radical powers $rad^{i}M$\ and the (higher) socle submodules $%
S_{i}\left( M\right) $\ of a module $M$\ are distinct and characteristic
submodules, giving by their successive (semisimple) quotients the well known
radical (or Loewy) and socle series of it.\ Although virtual radical and
socle series are obtainable through virtual diagrams, it turns out that
these important submodules play a crucial role in proving the existence of
virtual diagrams for arbitrary modules (see \cite{StG}).\ 

\textbf{Given a virtual diagram \ }$D\left( M\right) $\textbf{, the
corresponding virtual radical series is gotten by placing the simple virtual
constituents in layers and as high as possible, meaning that, while moving
along paths (of downward successive edges) from the head of }$M$\textbf{,
all edges shall be going from the one radical layer down to the next one,
unless the simple virtual on the lawer end of the edge is held lower down by
some longer vertical paths; dually for the socle series. We shall call that
characteristic peculiarity in the shaping out of the virtual radical, resp.
socle, series of }$M$\textbf{, corresponding to a virtual diagram of }$M$%
\textbf{, an \textit{upward or downward tropism}.}

Lemma \ref{ssQR} says indeed that we can on a virtual socle series of $M$
follow the virtual simple costituents of submodules of $M$, correspondingly
for quotient modules on a virtual radical series. This does on the other
hand no only become apparent under the light of the virtual category but
also is a special case of a much more general fact, namely that we can
locate any $\mathfrak{D}\left( M\right) $-visible virtual constituent on the
virtual diagram $\mathfrak{D}\left( M\right) $\ and follow it in the
generation of a virtual radical (resp., socle) series of $M$.\ 

\begin{definition}
\bigskip\ (i) Given a module $M$, we shall call a submodule or a quotient of 
$M$ a full-depth one, if it has the same Loewy length as $M$. \ 

(ii) Given an object $E$\ in $\mathfrak{V}\left( M\right) $, a virtual
radical series of $M$, manifesting $E$, so that the "lowest" in that radical
series of $M$ lying simple constituents of $E$ are on the $i$-th layer (i.e.
in $rad^{i-1}M\diagup rad^{i}M$), we define as a maximal-depth ($E$%
-)enclosing quotient of $M$\ one of Loewy length equal to $i$, where by
"enclosing" we bear in mind the ordering "$\lessdot $". Dually, we shall
likewise speak of a maximal depth $E$-enclosing submodule of $M$.
\end{definition}

It is trivial to check that the definitions in (ii) above are unambiguous in
the frame of the virtual category.

\begin{remark}
\label{imp}IMPORTANT:Whenever there are projective covers, resp. injective
hulls, in a category of modules (such as the one of [finitely generated]
modules over an Artinian ring), we may consider any given module $N$\ as a
quotient of a projective cover, resp. a submodule of an injective hull,
whenever of course $N$ is isomorphic to a quotient (resp., a submodule) of
the latter and corresponding to that - recall in this respect \cite[6.25(ii)]%
{CR} and its dual statement: That virtual identification means that, by "a
cover" or "a hull" we may demand one whose relevant part (be it at the
bottom in the injective hull or on top in the projective cover) is
"virtually" identifiable (rather than just isomorphic) to some given $N$,
thus not contradicting the "uniqueness" (i.e., up to isomorphism) of these
concepts in their general module category. We shall briefly designate them
as projective covers, resp. injective hulls, \textbf{virtual over} some
(virtually, i.e. not just up to isomorphism) given module, e.g. the virtual
portion of $N$\ in question. This remark is enhanced by the following
proposition (and its corollary).

There is of course some well-placed intuition involved here - and there does
certainly arise a challenge of setting up some better fitted
category-theoretic conceptual frame to match these subtleties. Our main
principle and concern in this article here (and the following one) is to
suit "the cloths to the person", not conversely: Meaning that, as long as
there is no mathematical ambiguity, we are primarily concerned with our
mathematical ideas and concerns.
\end{remark}

\bigskip

\begin{proposition}
\label{ipE}Given\ a virtual diagram $\mathfrak{D}$ of a module $M$, the
virtual object $E$ corresponding to a $\mathfrak{D}$-edge $%
\begin{array}{c}
A \\ 
\downarrow \\ 
B%
\end{array}%
$ may be viewed both as a submodule of an injective hull $I_{B}$ of $B$\
over some maximal depth ($E$-)enclosing quotient of $M$, or as a quotient of
a projective cover $P_{A}$ \ of $A$\ over some maximal depth ($E$-)enclosing
submodule of $M$. Furthermore in both cases that quotient or submodule may
be chosen to be "$\lessdot $"-maximal.
\end{proposition}

\begin{proof}
\bigskip Having a virtual diagram, we can get both a virtual radical and a
virtual socle series of $M$. Let $M\diagup rad^{t}M$\ be the least radical
quotient of $M$, containing our edge. By factoring out every direct summand
of $M\diagup rad^{t}M$, other than the one containing $B$, we get some new
quotient of $M$, which is in the same virtual class as the indecomposable
summand $N$ of $M\diagup rad^{t}M$, containing $B$\ (i.e., representing the
same virtual object as $N$\ does); it is clear that, the summand $N$
containing $B$\ also encloses $E$, otherwise we would get a contradiction to
the indecomposability of the latter. Consider that quotient, that contains
that edge, e.g. $E$\ as a submodule. Then, by taking the virtual injective
hull $I_{B}$ (of\ $B$) over $E$ (see remark above), we have an embedding (:
as a submodule!) $\phi $ of $E$\ in $I_{B}$, whence the injective property
yields an extension of $\phi $\ to $N$, implying the existence of a quotient 
$N^{\prime }$ of $N$\ ("over $E$") of Loewy length $t+1$, which is
isomorphic to a submodule of $I_{B}$. But then we may take as $I_{B}$\ an
injective hull over $N^{\prime }$\ instead. We can clearly choose the
extension of $\phi $\ so that $N^{\prime }$ be maximized. We dualize the
proof for the dual statement.

It is important to point out here that we may get to some such quotients
that are $\mathfrak{D}$-visible by means of lemma \ref{CL}; likewise for the
dual case too.
\end{proof}

\bigskip An apparent disadvantage of this proposition is that it presumes
the existence of a virtual diagram for the given module $M$, then it takes
not any subquotient of $M$ but an extension of $B$\ by $A$ in the virtual
category, that realizes an edge of that diagram. However, we are going to
show in our forthcoming work \cite{StG} that, not only every (always
finitely generated!) module has a virtual diagram, but that this may even be
one that manifests any (submodule of a) subquotient, that is an
indecomposable of composition length 2. Not having done this yet, we adjust
the above proof to an apparently more general assumption, to get the
following corollary. This subordination of statements (rather than
vice-versa) is justified by our main interest in the virtual diagrammatic
point of view:

\begin{corollary}
\label{extr}Let a module $M$ be given and a filtration $N\subset K\subset
L\subset M$, such that $L\diagup N$\ be an indecomposable of composition
length 2; call $A$, $B$, $E$ the objects of the virtual category $\mathfrak{V%
}\left( M\right) $, corresponding to the subquotients $L\diagup K$, $%
K\diagup N$, $L\diagup N$. Then $E$ may be viewed both as a submodule of an
injective hull $I_{B}$ over some maximal depth ($E$-)enclosing quotient of $%
M $,\ or as a quotient of a projective cover $P_{A}$ \ over some maximal
depth ($E$-)enclosing submodule of $M$. Furthermore in both cases that
quotient or submodule may be chosen to be "$\lessdot $"-maximal - and it is
indecomposable.
\end{corollary}

\begin{proof}
\bigskip For the first part, we may here substitute $M\diagup N$\ for the $N$%
\ of the last proof and imitate the arguments; dualize for the second
statement. Indecomposability becomes evident through the fact that the socle
(respectively, the head) is simple.
\end{proof}

Finally, we state a special case of the last proposition, indeed already
covered by the remark \ref{imp} preceding it, because of its importance in
what follows:

\begin{corollary}
\label{edge}Every edge of a virtual diagram of a module $M$\ (and its
corresponding virtual object $E$)\ may be virtually identifiable both, with
a quotient of the projective cover of its head or a submodule of an
injective hull of its socle.
\end{corollary}

\bigskip

\begin{lemma}
\label{1extr}That maximal quotient $F_{B}$ (resp., maximal submodule $G_{A}$%
) described in corollary \ref{extr} is independent of the chosen filtration
in the statement, as long as we retain the same simple virtuals $A$ and $B$.
\end{lemma}

\begin{proof}
The critical remark to make is, in the first case, that also in any other
suitable filtration $N^{\prime }\subset K^{\prime }\subset L^{\prime
}\subset M$, there may no simple virtual constituents of $F_{B}$, other than 
$A$\ and $B$, be enclosed by (necessarily then, all of) $N^{\prime }$, $%
K^{\prime }$, $L^{\prime }$,\bigskip as that would give a (generalized)
"blunted arm" as a submodule (in a suitable submodule of $F_{B}$), in
contradiction either to lemma \ref{comprL} or to proposition \ref{ReD}. Dual
arguments for the other case.
\end{proof}

\begin{definition}
\label{Dextr}We shall call call that maximal quotient $F_{B}$ (resp.,
maximal submodule) described by lemma \ref{1extr} the injective extract of $%
B $\ [over $A$] (resp., the projective extract of $A$\ [over $B$]) in $%
\mathfrak{V}\left( M\right) $.
\end{definition}

It is clear that the the injective extract $F_{B}$ of $B$\ over some other
virtual simple $A^{\prime }$, which is enclosed by $F_{B}$, is $F_{B}$\
again. But then by a combination of the fact, that a generalized arm
(outgoing from node $B$) cannot correspond to a quotient (lemma \ref{comprL}%
) and of proposition \ref{ReD} we draw the conclusion (and its dual) that
any such virtual $A^{\prime }$\ has to be enclosed by $F_{B}$! Notice that
we only need notions of some generalized diagrams corresponding to
filtrations, with no need of assuming (or proving) the existence of a
suitable virtual diagram for $M$.

That means that in the last definition we may in fact drop the references
"over $A$/over $B$" - and just speak of the injective (resp. projective)
extract of some simple virtual in $\mathfrak{V}\left( M\right) $.

\section{\protect\bigskip Extensions from a virtual point of view}

In what follows we shall often be referring to \cite[Chapter III, especially
sections 1, 3, 5.]{ML}. We shall look at (1-) extensions of $R$-modules, and
first of all we want to discuss the case of "proportional" extensions, as we
have called them in \cite{StG}, corresponding to "composites of a small
exact sequence with a homomorphism \textit{(in our case: an automorphism)}"
according to their exposition in \cite[Chapter III, section 1]{ML}:

\bigskip For $R$-modules $A$, $C$ we consider the set of extensions $%
YExt_{R}^{1}\left( C,A\right) $ as consisting of the equivalence classes in
the class of short exact sequences $\mathcal{E}$ (by which we may also
denote its extension class) of the form $\mathcal{E}:0\longrightarrow
A\longrightarrow E\longrightarrow C\longrightarrow 0$ in the usual manner
(the corresponding notation in \cite{ML} is $Ext_{R}\left( C,A\right) $).

$YExt_{R}^{1}\left( \_,\_\right) $ is a bifunctor, contravariant in the
first, covariant in the second argument: Let, so, $\alpha :A\longrightarrow
A^{^{\prime \prime }}$, $\gamma :C^{^{\prime }}\longrightarrow C$\ be
non-zero $R$-module homomorphisms; our bifunctor transforms them to $\alpha
_{\ast }:\mathcal{E}\longrightarrow \mathcal{E}^{^{\prime \prime }}$, $%
\gamma ^{\ast }:\mathcal{E}^{\prime }\longrightarrow \mathcal{E}$\ by using
identity map on the other extreme term and by getting the middle term of $%
\mathcal{E}^{^{\prime \prime }}$,\ $\mathcal{E}^{\prime }$ \ as\ push-out
and pull-back, respectively. We are also adopting MacLane's designation of
the resulting extensions $\mathcal{E}^{^{\prime \prime }}=\alpha _{\ast }%
\mathcal{E}$,\ $\mathcal{E}^{\prime }=\gamma ^{\ast }\mathcal{E}$ as,
respectively, $\alpha \mathcal{E}$\ and $\mathcal{E}\gamma $; we shall,
correspondingly, denote by $\alpha E$\ and $E\gamma $\ the middle terms
resulting from the original $E$. It is also important to understand $\alpha 
\mathcal{E}$, resp. $\mathcal{E}\gamma $,\ as the obstruction for extending $%
\alpha $\ to $E$, resp. for lifting $\gamma $\ to $E$, and the obstruction
homomorphisms $\tau ^{\ast }=\tau _{\mathcal{E}}^{\ast }:Hom_{R}\left(
C^{^{\prime }},C\right) \ni \gamma \longmapsto \mathcal{E}\gamma \in
YExt_{R}\left( C^{\prime },A\right) $, $\tau _{\ast }=\tau _{\ast \mathcal{E}%
}:Hom_{R}\left( A,A^{\prime \prime }\right) \ni \alpha \longmapsto \alpha
_{\ast }\mathcal{E=}\alpha \mathcal{E}\in YExt_{R}\left( C,A^{\prime \prime
}\right) $ as the connecting homomorphisms for the $\mathcal{E}$-derived
covariant, resp. contravariant, long exact sequences, see \cite[Chapter III,
section 3 \ - in particular lemmata 3.1 \& 3.3]{ML} - where also notably its
involved higher steps (or at least the first one, involving 1-extensions)
are realized in a somehow self-contained manner, i.e. just by using
extension classes of sequences, with no reference to the proper concept of
the derived functors: We have of course also the Yoneda identification of
the above mentioned functors $YExt$ with the derived functors $Ext$, see for
example \cite[III 2.4, IV 9.1]{ML}: "There is natural equivalence between
the set-valued bifunctors $YExt_{R}^{n}\left( C,A\right) $ (of equivalence
classes of exact $n$-sequences from $C$ to $A$)\textit{\ and }$%
Ext_{R}^{n}\left( M,A\right) $, $n=1,2,...$."

We are now going to specialize even more: Let $A^{^{\prime \prime }}=A$, $%
C^{\prime }=C$ be simple $R$-modules. Assume also that $\mathcal{E}$ is
non-split.\ Then $\alpha $, $\gamma $ are automorphisms (unless they are 0)
and\ the middle terms $E^{^{\prime \prime }}=\alpha E$,\ $E^{\prime
}=\allowbreak E\gamma $ of the (non-congruent to $\mathcal{E}$, unless those
automorphisms are the identity maps!) short exact sequences $\mathcal{E}%
^{^{\prime \prime }}=\alpha \mathcal{E}$,\ $\mathcal{E}^{\prime }=\mathcal{E}%
\gamma $\ are isomorphic to $E$ (5-lemma), where $E$ is uniserial of length
2, with series $%
\begin{array}{c}
C \\ 
A%
\end{array}%
$.\ The middle term of $\mathcal{E}\gamma $\ is the pull-back of $\gamma $
and $\sigma :E\longrightarrow C$\ in $\mathcal{E}$, while the middle term of 
$\alpha E$\ is the push-out of $\alpha $\ and the map $A\longrightarrow E$\
in $\mathcal{E}$.\ 

We have first to define "upper and lower proportionals", respectively of
type $\mathcal{E}\gamma $\ and $\alpha \mathcal{E}$, as two distinct
concepts; it is immediate to see that each of them defines an equivalence
relation: we shall accordingly speak of "upper" or "lower" proportionality.

\bigskip From the short exact sequence $\mathcal{E}$ above\ we get by
applying the functor $Hom\left( C,\_\right) $ the long exact sequence (\cite[%
III 3.4]{ML}) of abelian groups

$0\rightarrow Hom_{R}\left( C,A\right) \rightarrow Hom_{R}\left( C,E\right)
\rightarrow Hom_{R}\left( C,C\right) \rightarrow YExt_{R}^{1}\left(
C,A\right) \rightarrow YExt_{R}^{1}\left( C,E\right) \rightarrow ...$ ,
giving here the (mono-, if $C$ \ and $A$ are non-isomorphic)\ morphism $\tau
^{\ast }=\tau _{\mathcal{E}}^{\ast }:Hom_{R}\left( C,C\right)
\longrightarrow YExt_{R}^{1}\left( C,A\right) $, which is\ implemented by
the assignment $\gamma \mapsto \gamma ^{\ast }\mathcal{E=E}\gamma $ \textit{%
in the way mentioned in the previous paragraph}. By (a variant of) Schur's
lemma is $Hom_{R}\left( C,C\right) \cong \mathfrak{D}_{c}$, a division ring.
It is easily verified that $\left( \gamma \circ \gamma ^{\prime }\right)
^{\ast }=\gamma ^{\prime \ast }\circ \gamma ^{\ast }$, i.e. $\mathcal{E}%
\left( \gamma ^{\prime }\gamma \right) =\left( \mathcal{E}\gamma ^{\prime
}\right) \gamma $, and that $\mathcal{E}\left( id_{\gamma }\right) =\mathcal{%
E}$,\bigskip\ for any extension class $\mathcal{E}$\ in $YExt_{R}^{1}\left(
C,A\right) $, meaning that:

\begin{lemma}
\label{D-a}$\mathfrak{D}_{c}$ acts$\ $on $YExt_{R}^{1}\left( C,A\right) $\
from the right. In particular we get that $\left( \mathcal{E}\gamma \right)
\gamma ^{-1}=\mathcal{E}$ \ \textbf{(1)}.
\end{lemma}

\bigskip This action is in general explainable by means of pull-back; in
this case it turns out to give the same middle term $E\gamma $ (see below,
proposition \ref{sam}).

Dually, through application of the functor $Hom\left( \_,A\right) $\ on the
above short exact sequence $\mathcal{E}$ we get a long exact sequence (\cite[%
III 3.2]{ML}), which yields the monomorphism \ $\tau _{\ast }=\tau _{\ast 
\mathcal{E}}:Hom_{R}\left( A,A\right) \rightarrowtail YExt_{R}^{1}\left(
C,A\right) $, amounting in terms of $YExt_{R}^{1}\left( C,A\right) $ to what
has explicitly been explained above by the annotation $\alpha \mapsto \alpha
_{\ast }$, where again $Hom_{R}\left( A,A\right) \ $is a division ring $%
\mathfrak{D}_{A}$. Also in this case it turns out that the middle term of
the obtained extension is virtually the same as the old one (again by
proposition \ref{sam}).

\bigskip In that way we get two different actions\ on $YExt_{R}^{1}\left(
C,A\right) $, one by taking automorphisms on $C$, the other one through
automorphisms of\ $A$. By looking at the second action, giving rise to the
notion of lower proportional extensions of $C$ (compare also with \cite{StG}%
), let us consider the projective cover $P$ of $C$; let us now further
assume that $R$\ is\ Artinian. Then it is well known (see for example \cite[%
Cor. 2.5.4]{DB}) that we have the following isomorphisms of $R$-modules:

\textit{Given an Artinian ring }$R$\textit{, the }$R$\textit{-modules }$A$%
\textit{, }$M$\textit{\ where }$A$\textit{\ is simple, for }$n>0$\textit{\
we have } the following isomorphisms of $R$-modules:\textit{\ }$%
Ext_{R}^{n}\left( M,A\right) $ $\cong $\ $Hom_{R}\left( \Omega
^{n}M,A\right) $\textit{, }$Ext_{R}^{n}\left( A,M\right) \cong Hom_{R}\left(
A,\Omega ^{-n}M\right) $\textit{. }\textbf{In particular,}

\begin{remark}
\label{RNE}\textit{\ (i) We have an 1-1 correspondence (in fact, \textbf{a
natural equivalence of set-valued bifunctors}) between} $YExt_{R}^{1}\left(
C,A\right) $,\ $Hom_{R}\left( \Omega C,A\right) $ and

$Hom_{R}\left( C,\Omega ^{-1}A\right) $ (see for example \cite[III 6.4]{ML}%
), where, by\textit{\ further} \textit{assuming the existence of projective
covers and injective hulls in our category of modules,} $\Omega C=radP$ and $%
\Omega ^{-1}A=I/socI$, with $P$ the projective cover of $C$, $I$\ the
injective hull of $A$.

(ii) From now on we assume that in our category of modules there exist both
projective covers and injective hulls.
\end{remark}

That means that, by considering a virtual radical series of $P/rad^{2}P$, we
get an identification of $YExt_{R}^{1}\left( C,A\right) $\ with $%
Hom_{R}\left( \Sigma A,A\right) $, where $\Sigma A$ is\ the $A$-part of $P$%
's second radical layer. We are going to investigate how this identification
becomes a natural $\mathfrak{D}$-module isomorphism (where $\mathfrak{D}=%
\mathfrak{D}_{A}$), by viewing $YExt_{R}^{1}\left( C,A\right) $ as a (left) $%
\mathfrak{D}$-module, by virtue of the above mentioned "$\mathfrak{D}$%
-proportionality action" on $A$. Let $\Sigma A\cong A^{m}$, i.e. a sum of $m$
copies of the irreducible $A$; from now on we shall be looking at those
copies as virtually identifiable (\cite{StG}).

On the other hand it is clear that $Hom_{R}\left( \Omega C,A\right)
=Hom_{R}\left( radP,A\right) \cong Hom_{R}\left( radP/rad^{2}P,A\right)
\cong Hom_{R}\left( \Sigma A,A\right) \cong \mathfrak{D}^{m}$.

\bigskip We shall at first see, why we may view the proportionals of any
factor module $E\ $of $P$ of the type $%
\begin{array}{c}
C \\ 
A%
\end{array}%
$ as factor modules of $P$ again. So, we are looking at the short exact
sequence $0\rightarrow A\longrightarrow E\longrightarrow C\rightarrow 0$,
while viewing $A$,$\ E$ and, of course,$\ C$\ as virtually determined w.r.t. 
$P$. \ Denote by $\mathcal{E}$ the extension class of the above sequence,
then we will denote the extension module of $\gamma ^{\ast }\mathcal{%
E=E\gamma }$\ by $E\gamma $, that of $\alpha _{\ast }\mathcal{E=\alpha E}$\
by $\alpha E$.

\textit{From our virtual point of view this consideration of virtual middle
extension terms as factors of projective covers (or, dually, as submodules
of injective hulls) is formulated as the following concern, which is also
fundamental for our whole setup:}

\begin{problem}
The projective property of the projective cover $P$ of $C$\ guarantees that
any extension of $C$ by a simple $A$ may be realized as a quotient module of 
$P$. How can we get a "virtual" overview of those quotients, compared to the
respective extension classes?

Dually, the injective property of the injective hull $I$ of $A$\ ensures
that any extension of any simple module $C$\ by $A$\ may be viewed as a
submodule of $I$. How can we get a "virtual" overview of those submodules,
compared to the respective extension classes?
\end{problem}

\bigskip It is crucial to realize that in any such 1-extension the above
suggested epimorphism must in the virtual context be viewed as the
composition of the canonical epimorphism by the isomorphism (indeed
automorphism, as we are going to see) induced by an automorphism of $C$
(compare to lemma \ref{D-a}); we may likewise (dually) view the suggested
monomorphism into $I$ as the composition of a natural inclusion with an
isomorphism induced (through a \textit{virtual} push out) by an automorphism
of $A$.

This suggests that one and the same (virtually!) module, let us say here,
one with series $%
\begin{array}{c}
C \\ 
A%
\end{array}%
$, is not just assignable to one but to a whole family of ("proportional",
as we are going to see) extensions! We may even occasionally allow ourselves
to use different letters for the same module, inasmuch as it is drawn into
different extensions. In the following we want to show a kind of converse to
this assertion, namely that proportional extensions are realizable by 
\textbf{the same} (not just an isomorphic!) module: But such a statement may
only make sense inside the frame of a virtual category! \textit{This
suggests that the (proper to the case) virtual category is actually the
proper frame to consider extensions in.}

Let so $P=P_{C}$ be the projective cover of $C$\ as above, and apply $%
Hom_{R}\left( P,\_\right) $ on $\mathcal{E}$:\ Assume first that $A\ncong C$%
,\ implying that $Hom_{R}\left( P,A\right) $=0=$YExt_{R}^{1}\left(
P,E\right) $,\ which through the above application of $Hom_{R}\left(
P,\_\right) $ on $\mathcal{E}$ yields an $R$-module isomorphism (2) $%
Hom_{R}\left( P,E\right) \widetilde{\rightarrow }Hom_{R}\left( P,C\right) $%
,\ through the canonical $\sigma :E\twoheadrightarrow C$ in $\mathcal{E}$.
Notice that this relationship alone would in this case be enough to assure
the statement of the following lemma!

It is also clear that $Hom_{R}\left( P,C\right) \cong End_{R}C$, naturally
given by the assignment $End_{R}C\ni \gamma \longmapsto \gamma \circ \rho
\in Hom_{R}\left( P,C\right) $, where $\rho $\ is the canonical epimorphism $%
P\twoheadrightarrow C$, inducing $id_{C}:Hd\left( P\right) \longrightarrow C$%
, \textit{in virtual terms}. Referring to the virtual category, let us
consider the middle term $E$\ of $\mathcal{E}$\ as a quotient of $P$\ by a
certain submodule $L$. Then we fix the canonical epimorphism $%
P\twoheadrightarrow C\ $ as the one corresponding to $\mathcal{E}$,\ while
its "distortions" by automorphisms of $C$ shall give its "upper
proportional" extensions; we illustrate this discussion by the following
diagram, in which we are also using (1) from lemma \ref{D-a} (observe that,
as a consequence of that lemma, $\widetilde{\gamma ^{-1}}=\tilde{\gamma}%
^{-1} $):\ 

\bigskip $%
\begin{array}{ccccc}
&  &  &  & P_{{}}^{{}} \\ 
&  &  & \swarrow \tau _{_{{}}\frame{}} & \downarrow ^{\rho } \\ 
\mathcal{E}:A\frame{} & \rightarrowtail & E\frame{} & \frac{{}}{\sigma }%
\twoheadrightarrow & C_{{}}\frame{} \\ 
^{_{id_{A}}}\downarrow ^{{}} &  & ^{{}}\downarrow ^{\tilde{\gamma}^{-1}} & 
& ^{{}}\downarrow ^{\gamma ^{-1}} \\ 
\mathcal{E}\gamma :A\frame{}\frame{} & \rightarrowtail & E\gamma \frame{} & 
\twoheadrightarrow & C_{{}}\frame{}%
\end{array}%
$ \ \ \ \ We are pointing out that the (uniquely\bigskip

through $\mathcal{E}$\ and $P\twoheadrightarrow C$\ determined) homomorphism 
$\tau :P\rightarrow E$ has to be surjective, inasmuch as it cannot split, $%
\rho \neq 0$ and $E$\ is "of type" $%
\begin{array}{c}
C \\ 
A%
\end{array}%
$ \ and not semisimple.

But also in case $A\cong C$\ we can maintain uniqueness of a $\tau
:P\rightarrow E$\ attached to some $P\twoheadrightarrow C$\ by the
requirement that our \ $\tau :P\rightarrow E$ be surjective and induce the
identity map on the socle level of the $\tau $-induced isomorphism $P\diagup
\ker \tau \rightarrow E$. However that requirement is contained in the
definition of upper proportionality! That completes the proof of the
following lemma.

On the other hand, to view the situation in rather set-theoretical terms
too, the crucial observation to make here is that, not only is $\ker \left(
P\twoheadrightarrow E\gamma \right) =\ker \left( P\twoheadrightarrow
E\right) $ but, furthermore, the composite $P\twoheadrightarrow E\gamma $ is
"almost" the canonical epimorphism $P\twoheadrightarrow E$, since it twists
the canonical one by an automorphism $\widetilde{\gamma ^{-1}}=\tilde{\gamma}%
^{-1}$, coinduced by $\gamma ^{-1}$\ and inducing the identity on the socle $%
A$ of $E$, thus just meaning a "reordering" of the $A$-cosets in $E$: it is
still the same set! If we conversely start with an isomorphism from an
object $\mathcal{E}^{\prime }$ to another $\mathcal{E}$ in $%
YExt_{R}^{1}\left( C,A\right) $, with $C$, $A$\ simple modules, with the
identity map on $A$, then it induces some automorphism $\gamma ^{\prime }$
on $C$, so that $\mathcal{E}^{\prime }=\mathcal{E\gamma }^{\prime }$.

That means that:

\begin{lemma}
\label{Upr}An upper proportionality class in $YExt_{R}^{1}\left( C,A\right) $%
\ is "virtually determined" by a certain quotient of the projective cover of 
$C$, which in turn is of course determined by a certain submodule, to be
called the "proportionality kernel".
\end{lemma}

\bigskip

It turns out to be futile, if we attempt to virtually determine "lower
proportionality" classes in $YExt_{R}^{1}\left( C,A\right) $, as quotients
of the projective cover of $C$, again by just looking at obstructions.
Instead of that, we now turn toward the proof for the Yoneda equivalence $%
Yon:Ext\rightarrow YExt$:

Let us denote by $\Sigma A$ the $A$-part of the second radical layer $%
radP\diagup rad^{2}P$\ of $P$. The Yoneda correspondence establishes a
bijection between $YExt_{R}^{1}\left( C,A\right) $ and\ $Hom_{R}\left(
\Omega C,A\right) =Hom_{R}\left( radP,A\right) \cong $\ $Hom_{R}\left(
radP\diagup rad^{2}P,A\right) \cong $

$\cong Hom_{R}\left( \Sigma A,A\right) $. Let us look closer at how
extension classes in $YExt_{R}^{1}\left( C,A\right) $ may be identified by
homomorphisms $\widetilde{\alpha }:\Sigma A\rightarrow A$: Such a
homomorphism has to split, due to semisimplicity of $\Sigma A$. There is
therefore a well defined submodule $\widetilde{A}$ $\left( \cong A\right) $%
of $\Sigma A$, henceforth to be called \textit{the support of }$\widetilde{%
\alpha }$, such that (3) $\Sigma A=\ker \widetilde{\alpha }\oplus \widetilde{%
A}$\ - and so that we may identify $\widetilde{\alpha }$\ by the submodule $%
\widetilde{A}$ of $\Sigma A$ (in fact, also identifiable as $Coim\left( 
\widetilde{\alpha }\right) $) and its (isomorphic, since it is onto and we
are in an exact category) restriction $\mathrm{a}:\widetilde{A}\rightarrow A$
to it.

Let us again return to an epimorphism $\zeta :\Omega C=radP\rightarrow A$ of
kernel, say, $L_{\zeta }$: Then $\zeta $ factors through an induced
isomorphism $Coim\left( \zeta \right) \widetilde{\rightarrow }A$, where $%
Coim\left( \zeta \right) =\Omega C\diagup \ker \zeta =\Omega C\diagup
L_{\zeta }$, by means of the canonical $\xi :\Omega C\twoheadrightarrow
\Omega C\diagup L_{\zeta }$. Write also $radP\diagup rad^{2}P=\Sigma A\oplus
N$, $\rho :radP\rightarrow radP\diagup rad^{2}P$ for the canonical
epimorphism there and set $R^{\prime }:=\rho ^{-1}\left( N\right) $, then
take as $\widetilde{\alpha }:\Sigma A\rightarrow A$ that $\zeta $\
considered modulo $R^{\prime }$: inasmuch as $Coim\left( \zeta \right) $\
considered modulo $R^{\prime }$ is virtually nothing but the $Coim\left( 
\widetilde{\alpha }\right) =\widetilde{A}$ we have seen above, and this $%
\zeta $-induced isomorphism $Coim\left( \zeta \right) \widetilde{\rightarrow 
}A$, which $\zeta $\ factors through and which also (together with $%
Coim\left( \zeta \right) $, of course) determines $\zeta $\ completely.
Therefore $\zeta $ may naturally and unambiguously be identified by this $%
\widetilde{\alpha }:\Sigma A\rightarrow A$, which precisely induces $\mathrm{%
a}$\ on its "support" $\widetilde{A}$: if we call $p$\ the projection of $%
\Sigma A\ $onto\ $\widetilde{A}$\ along the decomposition (3), then we have $%
\widetilde{\alpha }=\mathrm{a}\circ p$ - and $\mathrm{a}=\widetilde{\alpha }%
|_{\tilde{A}}$.

Also, we may write $\xi =p\circ r$, where $r:\Omega C\longrightarrow \Omega
C\diagup R^{\prime }=\Sigma A$ is the canonical epimorphism. Hence is $\zeta
=\mathrm{a}\circ \xi =\mathrm{a}\circ p\circ r$.\ \ 

So the epimorphism $\zeta :\Omega C=radP\rightarrow A$ of kernel $L_{\zeta }$
induces $\mathrm{a}$\ on $\widetilde{A}$, and $\zeta $ may be viewed as
gotten through twisting the canonical epimorphism $\xi :radP\rightarrow 
\widetilde{A}$ (with $\widetilde{A}$\ virtually equal to $radP\diagup
L_{\zeta }$)\ by $\mathrm{a}$ (i.e. $\zeta =\mathrm{a}\circ \xi $),
similarly to what we have seen in the previous case.\ It is clear that $%
L_{\zeta }=$\ $\rho ^{-1}\left( N\oplus \ker \widetilde{\alpha }\right) $.
By defining the canonical epimorphism $\rho _{0}:P\rightarrow P\diagup
rad^{2}P$, we see that $L_{\zeta }=$\ $\rho _{0}^{-1}\left( N\oplus \ker 
\widetilde{\alpha }\right) $\ as well.

We identify for a moment $A$\ with $\widetilde{A}$, so that we may virtually
look at arbitrary lower proportionals.\ Set $E=P\diagup L_{\zeta }$, and
call $E_{\mathrm{a}}$\ its lower $\mathrm{a}$-proportional (to be denoted as 
$\mathrm{a}E$, like\ above), i.e. $E_{\mathrm{a}}$\ is a push-out, but a
virtual one, as in the diagram\ \ \ \ \ \ \ \ \ \ \ \ \ \ \ \ \ \ \ \ 

\bigskip

$%
\begin{array}{ccccc}
\mathcal{E}:\widetilde{A}\frame{} & \rightarrowtail & E\frame{} & \frac{%
\sigma }{{}}\twoheadrightarrow & C_{{}}\frame{} \\ 
\frame{}^{\mathrm{a}}\downarrow ^{{}} &  & ^{{}}\downarrow ^{{}} &  & 
^{{}}\downarrow ^{id_{C}} \\ 
\mathrm{a}\mathcal{E}:\widetilde{A}\frame{}\frame{}\frame{}_{{}} & 
\rightarrowtail & E_{\mathrm{a}}\frame{} & \twoheadrightarrow & C_{{}}\frame{%
}%
\end{array}%
$ . It is namely clear that the quotient of $P$,

\bigskip

in which the extension $\mathrm{a}\mathcal{E}$\ is realizable as precisely $%
P\diagup L_{\zeta }$, i.e. the same as with $\mathcal{E}$; indeed, in the
virtual category $\mathfrak{V}\left( P\right) $ of $P$ we may look at those
middle terms of extensions as quotients of that quotient $E_{12}$ of $P$,
which results from factoring out the submodule that is largest possible,
under the condition that both $E$\ and $E_{\alpha }$ still\ be quotients
(virtually meant) of that quotient. That quotient $E_{12}$ is easily seen to
be the pull-back of $E$\ and $E_{\mathrm{a}}$\ over the top $C$ in $%
\mathfrak{V}\left( P\right) $.

If $E$\ and $E_{\mathrm{a}}$\ were different quotients of $P$,\ then that
pull-back in the virtual category of $P$\ would have the diagrammatic form \ 
$%
\begin{array}{c}
\diagup \diagdown%
\end{array}%
$, with $C$ on top; however this is contradicted by the fact that the
pull-back of $E$\ and $E_{\mathrm{a}}$\ is just $E$, as it is immediately
verified by the universal property from the diagram:

\bigskip $%
\begin{array}{ccc}
&  &  \\ 
&  &  \\ 
&  & 
\end{array}%
\begin{array}{ccc}
E & \frac{id_{E}}{{}}\longrightarrow & E \\ 
\downarrow & \swarrow _{{}} & \downarrow \\ 
E_{\mathrm{a}}\  & \frac{{}}{{}}\longrightarrow & C%
\end{array}%
$

Therefore $E_{\mathrm{a}}$\ in $\mathrm{a}\mathcal{E}$ is realized as the
same quotient module of $P$, as $E=P\diagup L_{\zeta }$ in $\mathcal{E}$.

Seen in a more set-theoretic way,$\ E$\ is equal to \ $\tbigcup\limits_{x\in
C}\sigma ^{-1}\left( x\right) $, each $\sigma ^{-1}\left( x\right) $\ being
an $\widetilde{A}$-coset, while the meaning of the identity being induced in
the previous diagram, that shows the isomorphism $\mathcal{E}\longrightarrow 
\mathrm{a}\mathcal{E}$, is that $E_{\mathrm{a}}$ still consists of precisely
the same (set-theoretically!) cosets, where only they are "rearranged", by
application of the automorphism $\mathrm{a}^{-1}$\ on $\widetilde{A}$.

\bigskip We now illustrate the situation with the following diagram:

\bigskip $%
\begin{array}{cc}
&  \\ 
&  \\ 
&  \\ 
&  \\ 
&  \\ 
&  \\ 
&  \\ 
&  \\ 
&  \\ 
&  \\ 
& 
\end{array}%
$($D_{1}$) $\ 
\begin{array}{ccccccccc}
&  & 0 &  & 0 &  &  &  &  \\ 
&  & \downarrow &  & \downarrow &  &  &  &  \\ 
&  & L_{\zeta } & = & L_{\zeta } &  &  &  &  \\ 
&  & \downarrow &  & \downarrow &  &  &  &  \\ 
\frame{}_{{}}0 & \rightarrow & \Omega C & \longrightarrow & P & 
\longrightarrow & C_{{}} & \rightarrow & 0 \\ 
&  & _{{}}\downarrow ^{\xi } &  & \downarrow &  & _{{}}\frame{}\downarrow
^{_{id_{C}}} &  &  \\ 
\mathcal{E}:0 & \rightarrow & \widetilde{A} & \longrightarrow & E & 
\longrightarrow & C & \rightarrow & 0 \\ 
&  & \downarrow ^{\mathrm{a}} &  & \downarrow &  & _{{}}\frame{}\downarrow
^{_{id_{C}}} &  &  \\ 
\mathrm{a}\mathcal{E}:0 & \rightarrow & A & \longrightarrow & E_{\mathrm{a}}
& \longrightarrow & C & \rightarrow & 0 \\ 
&  & \downarrow &  & \downarrow &  & \downarrow &  &  \\ 
&  & 0 &  & 0 &  & 0 &  & 
\end{array}%
$ \ \ \ ($\zeta =\mathrm{a}\circ \xi $, $\ker \zeta =\ker \xi =L_{\zeta }$)

Notice that unless $A$\ be\ identified with $\widetilde{A}$,\ $\mathrm{a}%
\mathcal{E}$\ is\ somehow "abstract": we shall however use the above
discussion and diagram ($D_{1}$) as a kind of springboard, to enhance its
scope in a virtual\ direction:\ \ \ \ \ \ \ \ \ \ \ \ \ \ \ \ \ \ \ \ 

\bigskip We may identify the automorphisms of $\widetilde{A}$\ with those of 
$A$, by fixing an (arbitrary!) isomorphism; such an identification is
therefore only conceivable up to an automorphism of $A$. We shall call such
an automorphism-identifying isomorphism an \textit{identomorphism}. However
this relativity is a consideration that might only be taken when comparing
identomorphisms\ to the same $A$ of supports of different such epimorphisms $%
\Sigma A\rightarrow A$, while this issue remains anyway insignificant when
talking of "lower proportionals" of the same extension, as the approach is
then done by fixing just one such identomorphism.

That is, having established such an identification $j:\widetilde{A}%
\rightarrow A$ for our $\widetilde{A}$\ here, we may then view $\mathrm{a}$\
in the above diagram ($D_{1}$) as equal to $\alpha \circ j$, for an element $%
\alpha $\ of $AutA$, thus also identifying $j\mathcal{E}$\ with $\mathcal{E}$%
,\ so that we may now get all the lower proportionals of $\mathcal{E}$\
through automorphisms of $A$.

\bigskip\ \ \ $%
\begin{array}{cc}
&  \\ 
&  \\ 
&  \\ 
&  \\ 
&  \\ 
&  \\ 
&  \\ 
&  \\ 
&  \\ 
&  \\ 
& 
\end{array}%
$($D_{2}$) $%
\begin{array}{ccccccccc}
&  & 0 &  & 0 &  &  &  &  \\ 
&  & \downarrow &  & \downarrow &  &  &  &  \\ 
&  & L_{\zeta } & = & L_{\zeta } &  &  &  &  \\ 
&  & \downarrow &  & \downarrow &  &  &  &  \\ 
\frame{}_{{}}0 & \rightarrow & \Omega C & \longrightarrow & P & 
\longrightarrow & C_{{}} & \rightarrow & 0 \\ 
&  & _{{}}\downarrow ^{\xi } &  & \downarrow &  & _{{}}\frame{}\downarrow
^{_{id_{C}}} &  &  \\ 
\mathcal{E}:0 & \rightarrow & \widetilde{A} & \longrightarrow & E & 
\longrightarrow & C & \rightarrow & 0 \\ 
&  & \downarrow ^{j} &  & \downarrow &  & _{{}}\frame{}\downarrow
^{_{id_{C}}} &  &  \\ 
\mathcal{E}:0 & \rightarrow & A & \longrightarrow & E & \longrightarrow & C
& \rightarrow & 0 \\ 
&  & \downarrow ^{\alpha } &  & \downarrow &  & \downarrow &  &  \\ 
\alpha \mathcal{E}:0 & \rightarrow & A & \longrightarrow & E_{\alpha } & 
\longrightarrow & C & \rightarrow & 0 \\ 
&  & \downarrow &  & \downarrow &  & \downarrow &  &  \\ 
&  & 0 &  & 0 &  & 0 &  & 
\end{array}%
$

\bigskip

\begin{remark}
\label{lpr}It must be here observed that, in identifying an extension class
in $YExt_{R}^{1}\left( C,A\right) $ by means of an arbitrary epimorphism $%
\Sigma A\rightarrow A$, where $A$ is, say, a "prototypic isomorphic copy",
this may at first glance seem only to be possible up to lower
proportionality; however we shall show in the following subsection that this
seemingly innate relativity may easily be overcome. \ 
\end{remark}

We may now dualize everything above, by working on $E$\ as embedded in the
injective hull $I$\ of $A$;\ in particular, if $E\gamma $ were realized by a
different submodule of $I$, then the push-out of these two over $A$ would
contain them both properly: however it is easy to verify (again by the
universal property) that this push-out is actually $E$ itself. It may on the
other hand again be precisely seen, how any upper proportional $\mathcal{%
E\gamma }$\ of an extension $\mathcal{E}$ of $C$\ by $A$\ is also a lower
upper proportional $\alpha \mathcal{E}$\ of $\mathcal{E}$, for some
automorphism $\alpha $ of $A$.

We sum our results up (and the similarly deducible dual analogues) in the
following

\begin{proposition}
\bigskip \label{sam}A lower proportionality class in $YExt_{R}^{1}\left(
C,A\right) $\ is "virtually determined" by a certain quotient of the
projective cover $P$ of $C$, as its "virtual middle term", identifiable by a
direct summand of $radP\diagup rad^{2}P$ isomorphic to $A$; dually, an upper
proportionality class in $YExt_{R}^{1}\left( C,A\right) $\ is "virtually
determined" by a certain submodule of the injective hull \ of $A$ as its
"virtual middle term", identifiable by a direct summand of $soc^{2}I\diagup
socI$ isomorphic to $C$, where $I$\ is the injective hull of $A$.

There exists, further, a bijection \textit{between} $YExt_{R}^{1}\left(
C,A\right) $, $Hom_{R}\left( \Omega C,A\right) $\ and $Hom_{R}\left( \Sigma
A,A\right) $\ - and, similarly, a bijection between $YExt_{R}^{1}\left(
C,A\right) $,

\ $Hom_{R}\left( C,\Omega ^{-1}A\right) $ and\ $Hom_{R}\left( C,\Sigma
C\right) $ , where $\Omega C=radP$ and $\Omega ^{-1}A=I/socI$, with $P$ the
projective cover of $C$, $\Sigma A$ is the $A$-part of $radP\diagup rad^{2}P$%
, $\Sigma C$\ the $C$-part of $soc^{2}I\diagup socI$.
\end{proposition}

Now, by taking corollary \ into account, we see that in reality we may
deduce a much stronger conclusion, namely:

\begin{corollary}
Given a module $M$\ and an indecomposable submodule $E$\ of a subquotient of 
$M$,\bigskip having composition length 2, head $C$\ and socle $A$\
(virtually speaking), the virtual $E$ realizes both an upper and a lower
proportionality class of extensions of $A$\ by\ $C$:
\end{corollary}

We might finally as well demonstrate a more set-theoretic view of the
identical realizations of the extension modules $E$ and $E\gamma $:

\begin{lemma}
(i) The extension module $E\circ id_{C}\ $is identifiable with $E$. (ii) The
extension module $E\gamma $, for any $\gamma \in AutC$, viewed as a
submodule of $E\times C\ $is in fact $\left( E\circ id_{C}\right) ^{\left(
1,\gamma \right) }$ in the notation of \cite[subsection 4.2]{StG1}\ - and it
is virtually identifiable with $E$ in $P$, as the same quotient module.
\end{lemma}

\begin{proof}
\bigskip (i) We have to remind that the extension module $E\circ id_{C}\ $is
only defined up to isomorphism, not set-theoretically; it is also clear that
it is isomorphic to $E$, it is only its extension class that varies. Its
standard concrete construction is given as a submodule of $E\times C$, in
this\ case $=\tbigcup\limits_{c\in C}\left\{ \left( \sigma ^{-1}\left(
c\right) ,c\right) \right\} $, where $\sigma :E\rightarrow C$ in the exact
sequence, whose elements may be viewed as being the same as those of $E$,
just written with the superfluous suffix $c$. It is in fact easy to check
that they are naturally isomorphic.

On the other hand lemma \ref{D-a} guarantees that $\left( id_{C}\right)
^{\ast }\mathcal{E=E}\left( id_{C}\right) =\mathcal{E}$, therefore in the
frame of a virtual category the fact that $\mathcal{E}$\ still remains in
the same equivalence after appliance of $\left( id_{C}\right) ^{\ast }$,
notably also while inducing identity on $C$,\ it is clear that $\left(
id_{C}\right) ^{\ast }$\ also induces the identity map on $A$, meaning that,
in virtual terms, $\left( id_{C}\right) ^{\ast }E=E\left( id_{C}\right) $\
is identical with $E$.

(ii) Compared to this writing of $E=E\circ id_{C}$, $E\circ \gamma $ is
written as

$\tbigcup\limits_{c\in C}\left\{ \left( \sigma ^{-1}\left( c\right) ,\gamma
^{-1}c\right) \right\} =\left( E\circ id_{C}\right) ^{\left( 1,\gamma
\right) }$, according to \cite[Prop.36]{StG1}, which is identical to $E$ as
a subset of $E\times C$, but that construction being not virtually defined
anyway, that does not mean anything here.

We can similarly to (i) above see that $\left( E\circ \gamma _{1}\right)
\circ \gamma _{2}=$\ $E\circ \left( \gamma _{1}\circ \gamma _{2}\right) $\
also as sets, implying that the functor $YExt\left( \_,A\right) $ may also
give rise to a functor from modules to sets, the sets of the extension
modules.\ On the other hand it is immediate to check that $\left( \mathcal{E}%
\gamma _{1}\right) \gamma _{2}=$\ $\mathcal{E}\left( \gamma _{1}\gamma
_{2}\right) $\ and, similarly, that $\alpha _{1}\left( \alpha _{2}\mathcal{E}%
\right) =\left( \alpha _{1}\alpha _{2}\right) \mathcal{E}$ for $\alpha
_{1},\alpha _{2}\in EndA$.

We have clearly a bijection between $AutC$ and $Hom\left( P,C\right) $, $%
\gamma \longmapsto \overline{\gamma }:=\gamma \circ \tau $, with $\tau
=\sigma \circ \rho $, $\rho :P\twoheadrightarrow E$ canonical, in which we
may attach the canonical projection $\tau $ to $id_{C}$, that virtually
corresponds to the extension module $E$. But then $\overline{\gamma ^{-1}}%
=\gamma ^{-1}\circ \rho $\ \ factors through the extension module $E\circ
\gamma $ in virtual terms, hence we get an epimorphism $P\longrightarrow
E\circ \gamma $ which is equal to $\beta ^{-1}\circ \rho $ has the same
kernel as $\xi $.\ This proves that $E$\ and $E\circ \gamma $\ are virtually
the same factor modules in $P$.\ \ 
\end{proof}

\bigskip \bigskip

\section{$\mathfrak{D}$- or $\mathfrak{K}$-space of virtual extensions of
irreducibles}

\ \bigskip We wish now to enhance our analysis preceding proposition \ref%
{sam} above.

Let us so again consider $\Sigma A$ as there and let us choose some
decomposition $\Sigma A=\tbigoplus\limits_{i=1}^{m}A_{i}$,\ with $A_{i}$\
fixed; let also $s_{i}:A_{i}\rightarrow A$\ be some fixed isomorphisms
("identomorphisms"), with $A$\ a "prototypic isomorphic copy" for them all:
Notice that such a "prototypic isomorphic copy" of the $A_{i}$'s also allows
us now by virtue of the "identomorphisms" $s_{i}$ to "coordinate" (or
identify, if you prefer) their automorphisms. Denote by $p_{i}$, $i=1,...,m$%
, the standard projections of $\tbigoplus\limits_{i=1}^{m}A_{i}$. We have a
natural isomorphism between\ $Hom_{R}\left( \Omega C,A\right) $ and $%
Hom_{R}\left( \Sigma A,A\right) $, therefore we may identify any element of
the first by an element of the second - and vice versa. Finally we get

\ \ \ \ \ \ \ \textbf{(4)} $\ \ Hom_{R}\left( \Omega C,A\right) \cong
\tbigoplus\limits_{i=1}^{m}Hom_{R}\left( A_{i},A\right) $.

Let us also denote the family of lower proportionality classes in $%
YExt_{R}^{1}\left( C,A\right) $ by $\overline{YExt_{R}^{1}\left( C,A\right) }
$.

We are now ready to define addition inside $YExt_{R}^{1}\left( C,A\right) $;
let us first introduce some relevant notation: \ \ 

We shall represent the homomorphism $s_{i}\circ p_{i}:\Sigma A\rightarrow A$%
\ by a column of length $m$, having $1$ as the $i$'th and $0$ at all other
entries. It is clear that such a homomorphism $s_{i}\circ p_{i}$ corresponds
then to the "default"\ extension $\mathcal{E}_{i}$, that is realized as a
quotient $E_{i}$ of $P$, "virtually containing"\ $A_{i}$ and in which the
epimorphism is just the canonical one.

This quotient $E_{i}$, virtually corresponding to an edge $%
\begin{array}{c}
C \\ 
A_{i}%
\end{array}%
$, is the \textit{injective extract} of $A_{i}$, see definition \ref{Dextr}
(in the virtual category of the projective cover of $C$).

We shall\ consider all homomorphisms $\Sigma A\rightarrow A$ written up as

$\tsum\limits_{i=1}^{m}\alpha _{i}\circ s_{i}\circ p_{i}$=$\left[ 
\begin{array}{ccc}
\alpha _{1} & ... & \alpha _{m}%
\end{array}%
\right] \circ \left[ 
\begin{array}{c}
s_{1}\circ p_{1} \\ 
. \\ 
. \\ 
. \\ 
s_{m}\circ p_{m}%
\end{array}%
\right] $=$\left[ 
\begin{array}{ccc}
\alpha _{1} & ... & \alpha _{m}%
\end{array}%
\right] \circ \left[ 
\begin{array}{c}
1 \\ 
. \\ 
. \\ 
. \\ 
1%
\end{array}%
\right] $, which we shall then briefly denote by $\left[ 
\begin{array}{ccc}
\alpha _{1} & ... & \alpha _{m}%
\end{array}%
\right] $ or, even simpler, as $\left( \alpha _{1},...,\alpha _{m}\right) $,
where $\alpha _{i}\in End\left( A\right) $.\ 

For $\mathcal{E}\in YExt_{R}^{1}\left( C,A\right) $ denote so by $\phi
\left( \mathcal{E}\right) $ its attached homomorphism $\Sigma A\rightarrow A$%
, and let $\overline{\alpha }\left( \mathcal{E}\right) $\ be the element of $%
\left( End\left( A\right) \right) ^{m}$, corresponding to $\phi \left( 
\mathcal{E}\right) $ in the way just described, and thus attached to $%
\mathcal{E}$; for $\kappa $\ such a homomorphism $\Sigma A\rightarrow A$,
let $\sup \left( \kappa \right) $ ($\subset \Sigma A$) denote its support, $%
\overline{\alpha }_{\kappa }\in \left( End\left( A\right) \right) ^{m}$ its
attached $m$-tuple of $A$-endomorphisms\ and, by a slight slackness of
notation, $ex\left( \kappa \right) =ex\left( \overline{\alpha }_{\kappa
}\right) $ the corresponding extension. \ \ \ 

\begin{definition}
\label{exsum}For $\mathcal{E}_{1}$, $\mathcal{E}_{2}\in YExt_{R}^{1}\left(
C,A\right) $, define their sum $\mathcal{E}_{1}+\mathcal{E}_{2}$ as the
extension class $ex\left( \phi \left( \mathcal{E}_{1}\right) +\phi \left( 
\mathcal{E}_{2}\right) \right) $.
\end{definition}

\bigskip Notice that our definition may be viewed as a virtualized analogue
to, but it may not be compared with the well known Baer sum, as these are
two different things: That is, the one is a formal and abstract
construction, the other is virtual.

\begin{example}
\label{exS}Let us look at some easy examples: Let $m=2$, $\alpha _{1}=\alpha
_{2}=id_{A}$ and set $\mathcal{E}_{1}=ex\left( \alpha _{1},0\right) $, $%
\mathcal{E}_{2}=ex\left( 0,a_{2}\right) $;\ then the middle term of the sum $%
ex\left( \phi \left( \mathcal{E}_{1}\right) \right) +ex\left( \phi \left( 
\mathcal{E}_{2}\right) \right) $ corresponds to the support gotten by the
(split) factoring out of the kernel of $\left( \alpha _{1},\alpha
_{2}\right) =\left( id_{A_{1}},id_{A_{2}}\right) $ on $\Sigma A=A_{1}\oplus
A_{2}$, reminding of taking the codiagonal $\triangledown _{A}$\ in the Baer
definition. If we generalize to the similar example for an arbitrary $m$, we
shall then have to (split-)outquotient on the $\mathfrak{D}_{A}$-space $%
\Sigma A$ (where $\mathfrak{D}_{A}=Hom_{R}\left( A,A\right) $) the
"hyperplane" defined by the equation $a_{1}+...+a_{m}=0$, where the choice
of coordinate system corresponds to choosing the summands $A_{i}$ of\ $%
\Sigma A$\ together with the choice of the $m$\ identomorphisms $%
s_{i}:A_{i}\rightarrow A$.\ \ \ 
\end{example}

We point out that this analysis actually takes place in a subdirect product:
Indeed, by recalling the ordering \ref{VO} on the family of submodules of
subquotients of $P$, let $U$\ be defined as the slimmest quotient of $P$, so
that $\Sigma A$\ be contained in it as a submodule. It is quite clear that $%
U $\ is\ the virtual pull-back of $E_{i}$'s pull-back over (their canonical
epimorphisms onto)\ $C$. In that way we somehow turn back to the beginning
of this article.

\begin{remark}
We might as well have approached the issue dually, by considering the
extension classes $YExt_{R}^{1}\left( C,A\right) $ embedded in the injective
hull of $A$, identifiable then so by homomorphisms in $Hom_{R}\left(
C,\Omega ^{-1}A\right) $, as we have seen.\ \ 
\end{remark}

\bigskip

\bigskip It is clear that, for any given extension $\mathcal{E}\in
YExt_{R}^{1}\left( C,A\right) $, $\alpha \mathcal{E}$ is realized by the
same quotient $E$ of $P$\ as $\mathcal{E}$\ does. Now we shall prove the
statement of lemma \ref{Upr} above anew, by involving the machinery that we
have developed:

\begin{lemma}
\bigskip For any $\mathcal{E}\in YExt_{R}^{1}\left( C,A\right) $ and any $%
\gamma \in AutC$, $\mathcal{E}\gamma $ is also realizable by the same factor
module $E$\ in $P$.
\end{lemma}

\begin{proof}
\bigskip In view of proposition \ref{sam}, as well as definition \ref{exsum}
and relation (4) it suffices to prove the claim for an extension having one
of the fixed $A_{i}$'s as its support, therefore with one of the $E_{i}$'s
as its realizing factor module; without loss of generality we may assume
that its support is $A_{1}$, i.e. that our extension is $\alpha \mathcal{E}%
_{1}$,\ for some $\alpha \in AutA$. Now we may take the injective hull $I$\
of $A$\ in such a way, that $E_{1}$\ is virtually embedded in it as a
submodule. If now $\alpha \mathcal{E}_{1}\gamma $ had another support than $%
A_{1}$\ in $P$, then that should also be the case inside $I$, as
multiplication $\left( \alpha \mathcal{E}_{1}\right) \gamma $\ is \textbf{%
similarly} defined in both cases. But that is a contradiction.
\end{proof}

\bigskip Thus we may in fact simply speak of proportionality classes in $%
YExt_{R}^{1}\left( C,A\right) $ and of their virtual realizations in the
virtual category $\mathfrak{V}\left( P\right) $; however we have still two
dual descriptions of those classes, in a way that becomes somewhat more
clarified in the following theorem.

Set again $Hom_{R}\left( C,C\right) \cong \mathfrak{D}_{c}$. Then we shall
prove that the mentioned bijection between certain homomorphisms and the
extension classes $YExt_{R}^{1}\left( C,A\right) $, is in reality a $%
\mathfrak{D}_{A}$-$\mathfrak{D}_{C}$-bimodule isomorphism, in a way that
gives us a very concrete virtual insight into the module structure of $%
YExt_{R}^{1}\left( C,A\right) $:\ \ \ 

\begin{theorem}
\label{YonHom}$\mathfrak{a}$. $YExt_{R}^{1}\left( C,A\right) $ has a $%
\mathfrak{D}_{A}$-$\mathfrak{D}_{C}$-bimodule structure.

$\mathfrak{b}$. There exists a (left) $\mathfrak{D}_{A}$-module isomorphism 
\textit{between} $YExt_{R}^{1}\left( C,A\right) $\ and

$Hom_{R}\left( \Omega C,A\right) $ and, similarly, a (right) $\mathfrak{D}%
_{c}$-module isomorphism \textit{between} $YExt_{R}^{1}\left( C,A\right) $
and\ $Hom_{R}\left( C,\Omega ^{-1}A\right) $, where $\Omega C=radP$ and $%
\Omega ^{-1}A=I/socI$, with $P$ the projective cover of $C$, $I$\ the
injective hull of $A$. These module isomorphisms give rise to natural
equivalences of the corresponding module valued bifunctors (compare with \ref%
{RNE}). \ \ 

$\mathfrak{c}$. $\overline{YExt_{R}^{1}\left( C,A\right) }$ has\ the the
structure of a $\mathfrak{D}_{A}$-projective space and it is isomorphic to $%
\emph{P}^{m-1}\left( \mathfrak{D}_{A}\right) $.\ Dually, the family $%
\overline{\overline{YExt_{R}^{1}\left( C,A\right) }}$ of upper
proportionality classes in $YExt_{R}^{1}\left( C,A\right) $ has the
structure of a $\mathfrak{D}_{C}$-projective space and it is isomorphic to $%
\emph{P}^{m-1}\left( \mathfrak{D}_{C}\right) $.
\end{theorem}

\begin{proof}
\bigskip \bigskip In view of definition \ref{exsum} and relation (4), it
suffices to prove the $\mathfrak{D}_{A}$-module pseudo-distributivity
properties for extensions having the $A_{i}$'s as their support.

The properties $id_{A}\mathcal{E=E}$, $\sigma \left( \mathcal{E}_{1}+%
\mathcal{E}_{2}\right) =\sigma \mathcal{E}_{1}+\sigma \mathcal{E}_{2}$, $%
\sigma \mathcal{E}+\tau \mathcal{E}=\left( \sigma +\tau \right) \mathcal{E}$%
\ are then a direct consequence of definition \ref{exsum}, where in the last
we have also to notice that the extension classes on both sides of the
equality indeed have the same support.\ \ 

We use the dual arguments for the right $\mathfrak{D}_{C}$-module structure.

As for the left \& right pseudoassociativity and the blended one, $\left(
\alpha \mathcal{E}\right) \gamma =\alpha \left( \mathcal{E}\gamma \right) $,
we refer to \cite[Ch. III, lemmata 1.2, 1.4, 1.6]{ML}.

Regarding ($\mathfrak{c}$), it is clear that "muliplying" from left with an $%
\alpha \in AutA$, correspondingly from the right with a $\gamma \in AutC$,
doesn't change the proportionality class. The rest is easily checked.
\end{proof}

\begin{remark}
\bigskip \bigskip This theorem is reminiscent of the known Auslander-Reiten
formula (see for example \cite[Th. 3.4.1]{LAH} or the original paper \cite%
{AR}): However it has to be pointed out that our context here (regarding
extensions, their addition \& module structure) is very different, being a
virtual one. On the other hand we are giving here this subject only the
scope, that may serve us to achieve our next step: That is, virtual diagrams
of modules, see \cite{StG}.
\end{remark}

We wish to close here by specializing the last theorem to a very usual case,
in which $R$\ shall be a finite dimensional $\mathfrak{K}$-algebra, where $%
\mathfrak{K}$\ is an algebraically closed field. Before doing that, we go
through a quick review of some basic relevant facts.\ \ 

Let us begin by looking at the $R$-endomorphisms of $\Sigma A\cong A^{m}$, $%
A $\ a simple $R$-module, beginning in a more general context: \ $%
End_{R}\left( A^{m}\right) \cong M_{m}\left( \mathfrak{D}\right) $, the ring
of $m\times m$ matrices with entries from $\mathfrak{D=}End_{R}\left(
A\right) $.

It is an immediate consequence of the Density Theorem that, if $M$\ is a
semisimple $R$-module, which is finitely generated as an $End_{R}\left(
M\right) $-module, then the canonical homomorphism $R\longrightarrow $\ $%
End_{End_{R}\left( M\right) }\left( M\right) $ is surjective. If $A$\ is a
simple $R$-module, then as a (finitely generated) module over the division
ring $\mathfrak{D}=End_{R}\left( A\right) $\ it must be free, say $\cong 
\mathfrak{D}^{\lambda }$, therefore is $End_{\mathfrak{D}}\left( A\right)
\cong M_{\lambda }\left( \mathfrak{D}^{op}\right) $ and we get the canonical
surjective homomorphism $R\longrightarrow End_{\mathfrak{D}}\left( A\right)
\cong M_{\lambda }\left( \mathfrak{D}^{op}\right) $, which in case $A$\ is
also a faithful $R$-module (i.e. by substituting the appropriate $R$-block,
in this case just meaning the appropriate simple summand of $R$, for $R$)
becomes an isomorphism. Notice that here, although we haven't assumed
semisimplicity of the ring $R$, it is its faithful action on a simple module
that implies its simplicity.

If we now take $R$\ to be a $\mathfrak{K}$-algebra, $\mathfrak{K}$\ a field,
then is $End_{R}\left( M\right) $\ a $\mathfrak{K}$-algebra too, therefore
is $M$\ a $\mathfrak{K}$-vector space. In case $\dim _{\mathfrak{K}}M<\infty 
$, then is $\dim _{\mathfrak{K}}End_{R}\left( M\right) <\left( \dim _{%
\mathfrak{K}}M\right) ^{2}<\infty $\ too, as $End_{R}\left( M\right) \subset
End_{\mathfrak{K}}\left( M\right) $.

By assuming further $R$\ to be a finite dimensional $\mathfrak{K}$-algebra
(notably also implying that it is Artinian), every simple $R$-module $A$,
being an $R$-epimorphic image of $R$, shall necessarily be finite $\mathfrak{%
K}$-dimensional. Since $\mathfrak{K}\cong \mathfrak{K}\cdot id_{A}\subset
End_{R}\left( A\right) =\mathfrak{D}$, the assumed finite $\mathfrak{K}$%
-dimensionality of $A$\ implies certainly that $A$\ is finitely generated as
a $\mathfrak{D}$-module; but then, as we have seen, the canonical
homomorphism $R\longrightarrow $\ $End_{\mathfrak{D}}\left( A\right) $ is
surjective.

If we also assume that $\mathfrak{K}$\ is algebraically closed then, by the
well known Schur lemma, $\mathfrak{D}=End_{R}\left( A\right) =$\ $\mathfrak{K%
}\cdot id_{A}\cong \mathfrak{K}$.

\bigskip We have not started from this assumption, while it is very
essential from our point of view to look at the action of $End_{R}\left(
A\right) $\ ($\cong \mathfrak{K}$ in this last case) on $YExt_{R}^{1}\left(
C,A\right) $. However in this case the last theorem becomes:

\begin{corollary}
Assume $R$\ to be a finite dimensional $\mathfrak{K}$-algebra, with $%
\mathfrak{K}$\ an algebraically closed field. With the same notation as
above, we then have:

$\mathfrak{a}$. $YExt_{R}^{1}\left( C,A\right) $ has a $\mathfrak{K}$%
-bimodule structure.

$\mathfrak{b}$. There exists a (left) $\mathfrak{K}$-isomorphism \textit{%
between} $YExt_{R}^{1}\left( C,A\right) $\ and

$Hom_{R}\left( \Omega C,A\right) $ and, similarly, a (right) $\mathfrak{K}$%
-module isomorphism \textit{between}

$YExt_{R}^{1}\left( C,A\right) $ and\ $Hom_{R}\left( C,\Omega ^{-1}A\right) $%
, where $\Omega C=radP$ and $\Omega ^{-1}A=I/socI$, with $P$ the projective
cover of $C$, $I$\ the injective hull of $A$. Furthermore, there exist
natural equivalences of the corresponding $\mathfrak{K}$-module valued
bifunctors.\ \ 

$\mathfrak{c}$. $\overline{YExt_{R}^{1}\left( C,A\right) }$ has\ the the
structure of a $\mathfrak{K}$-projective space and it is isomorphic to $%
\emph{P}^{m-1}\left( \mathfrak{K}\right) $.\ Dually, the family $\overline{%
\overline{YExt_{R}^{1}\left( C,A\right) }}$ of upper proportionality classes
in $YExt_{R}^{1}\left( C,A\right) $ has the structure of a $\mathfrak{K}$%
-projective space and, as such, it is isomorphic to $\emph{P}^{m-1}\left( 
\mathfrak{K}\right) $.
\end{corollary}

\bigskip\ Notice that, as long as we have not "coordinated"/interrelated the
left with the right action of $\mathfrak{K}$\ on $YExt_{R}^{1}\left(
C,A\right) $, we have to be careful with the non-commutable $\mathfrak{K}$%
-bimodule structure of the latter.

\bigskip\ 

\ \ \ \ \ \

\end{document}